\documentclass[10pt,twoside]{amsart}

\usepackage[english]{babel}
\usepackage[all]{xy}
\usepackage{amssymb,amsmath,bbm,url,xr,a4wide}
\usepackage[mathscr]{euscript}
\usepackage{mathrsfs}

\title{Orientable homotopy modules}

\author{Fr\'ed\'eric D\'eglise}
\address{LAGA\\
CNRS~(UMR 7539)\\
Universit\'e Paris~13\\
\hbox{Avenue~Jean-Baptiste~Cl\'ement}\\
93430 Villetaneuse\\France}
\email{deglise@math.univ-paris13.fr}
\urladdr{http://www.math.univ-paris13.fr/~deglise/}
\thanks{Partially supported by the ANR (grant No. ANR-07-BLAN-042)}
\date{February 2010.}


\newtheorem*{thmi}{Theorem}

\newtheorem{thm}{Theorem}[subsection]
\newtheorem{prop}[thm]{Proposition}
\newtheorem{cor}[thm]{Corollary}
\newtheorem{lm}[thm]{Lemma}

\theoremstyle{definition}
\newtheorem{df}[thm]{Definition}
\newtheorem{ex}[thm]{Example}
\newtheorem{num}[thm]{}

\theoremstyle{remark}
\newtheorem{rem}[thm]{Remark}

\numberwithin{equation}{thm}

\newcommand{\ab}{\mathscr Ab}
\newcommand{\T}{\mathscr T}

\newcommand{\sm}{Sm_k}
\newcommand{\smc}{Sm^{cor}_k}
\newcommand{\smp}{Sm_{k,\bullet}}

\newcommand{\SH}{SH(k)}

\newcommand{\Sp}{Sp^\Sigma(k)}
\newcommand{\DMgm} {DM_{gm}\!(k)}
\newcommand{\DM} {DM(k)}
\newcommand{\DMe} {DM^{eff}(k)}

\newcommand{\hmod} {\Pi_*(k)}
\newcommand{\hmor} {\Pi_*^{\eta=0}(k)}
\newcommand{\hmtr} {\Pi_*^{tr}(k)}

\newcommand{\modl} {\mathscr{M}Cycl(k)}
\newcommand{\pmodl} {\ab^{\ecorps}}
\newcommand{\ecorps} {\tilde{\mathscr E}_k}
\newcommand{\model} {\mathcal M^{sm}}

\newcommand{\smod}[1]{#1\!-\!mod}
\newcommand{\wmod}[1]{#1\!-\!mod^w}
\newcommand{\spmod}[1]{#1\!-\!mod^{Sp}}

\newcommand{\pro}[1] {\mathrm{pro}\!-\!{#1}}

\newcommand{\sus}{\Sigma^\infty} 

\newcommand{\ecuppx}[1] { \scriptstyle \ \boxtimes_{#1} \textstyle }

\DeclareMathOperator{\ohtr}{\otimes^{Htr}}

\newcommand{\ilim}[1] { \underset{#1}{\varinjlim} \ }

\newcommand{\pplim}[1] { \underset{#1}{"{\varprojlim}"} \ }
\newcommand{\pprod}[1] { \underset{#1}{"{\prod}"} \ }

\newcommand{\drap}[1] {\mathcal D \left( #1 \right)}

\DeclareMathOperator\pic{Pic}
\newcommand{\spec}[1] {\operatorname{\mathrm{Spec}}(#1)}

\DeclareMathOperator{\Hom}{Hom}
\DeclareMathOperator{\uHom}{\underline{Hom}}

\newcommand{\nis}{\mathrm{Nis}}
\newcommand{\nspi}[3]{\pi_{#1}(#3)_{#2}} 

\newcommand{\pFD}[1]{$(\mathbf{FD}_{#1})$}
\newcommand{\pC}[1]{$(\mathbf{C}_{#1})$}

\newcommand{\NN} {\mathbb N}
\newcommand{\ZZ} {\mathbb Z}
\newcommand{\QQ} {\mathbb Q}
\renewcommand{\AA} {\mathbb A}
\newcommand{\GG} {\mathbb G_m }
\newcommand{\PP} {\mathbb P}
\newcommand{\cO}{\mathcal O}
\newcommand{\cI}{\mathcal I}

\newcommand{\E}{\mathbb E}
\newcommand{\F}{\mathbb F}
\newcommand{\HH}{\mathbf H}
\newcommand{\uHH}{\underline{\mathbf H}}

\newcommand{\R}{\mathbf R}
\newcommand{\uR}{\underline{\mathbf R\!}\,}
\newcommand{\MGL}{\mathbf{MGL}}
\newcommand{\uMGL}{\underline{\mathbf{MGL}}}

\newcommand{\spi}[3]{\underline \pi_{#1}(#3)_{#2}} 
\newcommand{\uK}{\underline K^M}

\begin{document}

\maketitle

\tableofcontents

\begin{abstract}
We prove a conjecture of Morel identifying
 Voevodsky's homotopy invariant sheaves with transfers
 with spectra in the stable homotopy category 
 which are concentrated in degree zero for the homotopy t-structure
 and have a trivial action of the Hopf map.
 This is done by relating these two kind of objects to Rost's cycle modules.
 Applications to algebraic cobordism and construction of cycle classes are given.
\end{abstract}

\section*{Introduction}

In \cite{Mor1} and \cite{Mor2},
 F.~Morel started a thorough analysis of the stable homotopy category
 of schemes over a field culminating by computing the zero-th stable
 homotopy groups of the zero sphere $\nspi 0 * {S^0}$
  in a joint work with M.~Hopkins. 
 The result involves a mixture of the Witt ring of the field
 and its Milnor K-theory.
This paper is built on the idea of Morel that
 the Witt part contains the obstruction to orientation in stable homotopy.
 Indeed, let us recall that $\nspi 0 * {S^0}$ is generated
 as a graded abelian group by units of the field $k$
 and the Hopf map\footnote{See paragraph \ref{num:Hopf}.} $\eta$.
 As in topology, the Hopf map $\eta$ is the first obstruction
 for a ring spectrum to be oriented.\footnote{See remark \ref{rem:mhor&P^1}.}

On the other hand the units of $k$ generates a subring of
 $\nspi 0 * {S^0}$ which turns out to be exactly the Milnor
 K-theory of $k$, or, in contemporaneous terms,
 the part of motivic cohomology where the degree is equal to the twist.
 It is understand now, though still conjectural,
 that motivic cohomology
 is the universal oriented cohomlogy with additive formal group law.

\bigskip

To build something out of these general remarks,
 one has to go more deeply into the motivic homotopy theory.
 The category $\DMe$ of V.~Voevodsky's motivic complexes over a perfect field $k$
 is built out
 of the so called homotopy invariant sheaves with transfers 
 who have the distinctive property that their cohomology is
 homotopy invariant.
 
Starting from the point that they form
 the heart of a natural t-structure on $\DMe$,
 Morel introduced the homotopy t-structure on the stable homotopy
 category $\SH$, analog of the previous one,
 and identified its heart as some graded sheaves without transfers
 but still with homotopy invariant cohomology.
 They are called \emph{homotopy modules} by Morel
 (see definition \ref{df:htp_modl}) and we denote their category
 by $\hmod$.
These sheaves have to be thought as \emph{stable homotopy groups}:
 in fact, taking the zero-th homology of $S^0$ with respect
 to the homotopy t-structure gives a homotopy module $\spi 0 * {S^0}$
 whose fiber at the point $k$ is precisely the graded
 abelian group $\nspi 0 * {S^0}$.
 Note that, as a consequence,
  any homotopy modules has a natural action of
  the Hopf map $\eta$ seen as an element of  $\spi 0 * {S^0}(k)$.

To relate these two kind of sheaves, it is more accurate
 to introduce the non effective (\emph{i.e.} \emph{stable})
 version of $\DMe$, simply denoted by $\DM$.
 The canonical t-structure on $\DMe$ extends to a t-structure
 on $\DM$ whose heart $\hmtr$ is now made by
 some graded homotopy invariant sheaves
 with transfers which we call \emph{homotopy modules with
  transfers} (see definition \ref{df:hmtr}).
  
Then the naural map $\DM \rightarrow \SH$ induces a natural functor
$$
\gamma_*:\hmtr \rightarrow \hmod
$$
which basically do nothing more than forgetting the transfers.
In this article,
 we prove the following conjecture of Morel:
\begin{thmi}
\begin{enumerate}
\item The functor $\gamma_*$ is fully faithful 
and its essential image consists of the homotopy modules
on which $\eta$ acts as $0$.
\item The functor $\gamma_*$ is monoidal
 and homotopy modules with transfers can be described as homotopy modules
 with an action of the unramified Milnor K-theory sheaf.
\end{enumerate}
\end{thmi}
Actually the first part was conjectured by Morel
 (\cite[conj. 3, p. 71]{Mor1})
 and the second one is a remark we made, deduced from the first one.

Let us now come back to the opening point of this introduction:
 it turns out from this theorem that,
 for homotopy modules seen as objects of $\SH$,
 $\eta$ is precisely the obstruction to be orientable
 (see corollary \ref{cor:carac_mhtpo}
  for the precise statement).
 Moreover, the Milnor part of $\spi 0 * {S^0}$
 -- \emph{i.e.} the unramified Milnor K-theory sheaf --
 is the universal homotopy module on which $\eta$ acts as $0$.

\bigskip

This theorem relies on our previous work on homotopy modules with
 transfers where we show that they are completely determined
 by the system of their fibers over finitely generated extensions of $k$,
 which can be described precisely as a cycle module in the
 sense of M.~Rost. Actually, according to \cite[th. 3.4]{Deg9},
 this defines in fact an equivalence of categories: this gives a way
 to construct transfers on sheaves out of a more elementary algebraic
 structure on their fibers. 
 The proof of our theorem thus consists to show that the fibers of
 a homotopy module with trivial action of $\eta$ gives a cycle module;
 this was actually the original form of the conjecture of Morel.

This is done by appealing to our work on the Gysin triangle
 \cite{Deg8} in the framework of modules over ring spectra: 
 indeed, the main idea of the proof is that
 a homotopy module with trivial action of $\eta$ has a (weak)
 structure of a module over the motivic cohomology ring spectrum $\HH$.
After having recalled the central definitions involved in the formulation
 of the previous theorem in section \ref{sec:Morel_conj},
 we have dedicated section \ref{sec:preparation} to recall
 the main technical tools which will be involved in the proof:
 cycle modules, modules over ring spectra, Gysin triangles and morphisms,
 the coniveau spectral sequence and the computation of its diffetential.
 This enables us, in section \ref{sec:proof},
 to give a neat proof of the first part of our theorem though it uses
 all the previous technical tools. In the last section,
  we prove the additional fact made by the point (2) in the above
  and we collect a few additional ones which illustrate our result.
 Notably, we use the coniveau spectral sequence to obtain the following
 results:
\begin{itemize}
\item (Cor. \ref{cor:MGL_CH}) 
Let $\MGL$ be Voevodsky's algebraic cobordism spectrum.
Then for any smooth connected scheme of dimension $d$,
 the canonical map
$$
\MGL^{2d,d}(X) \rightarrow CH_0(X),
$$
with target the $0$-cycles modulo rational equivalence,
 is an isomorphism.
\item (Prop. \ref{prop:eta_cycle_class}) 
Let $\E$ be a monoid in $\SH$ satisfying the following assumptions:
\begin{enumerate}
\item[(a)] $\eta$ acts trivially on $\spi 0 * \E$.
\item[(b)] For any negative integers $n,m$, any smooth connected scheme $X$,
 and any cohomological class $\rho \in \E^{n,m}(X)$,
 there exists a non empty open $U \subset X$ such that $\rho|_U=0$.
\end{enumerate}
Then $\E$ admits an orientation whose associated formal group law
 is additive.
 Moreover, for any integer $n \geq 0$ and any smooth scheme $X$,
 there exists a canonical cycle class map
$$
CH^n(X) \rightarrow \E^{2n,n}(X)
$$
which satisfies all the usual properties.
\end{itemize}

\section*{Notations and conventions}

In this article, $k$ is a perfect field.
Any scheme is assumed to be a \emph{separated $k$-scheme}
 unless stated otherwise.
 Such a scheme will be called a smooth scheme if it is \emph{smooth of finite type}.
 We denote by $\sm$ (resp. $\smp$) the category of smooth schemes
 (resp. pointed smooth schemes).
 Given a smooth scheme $X$,
 we put $X_+=X \sqcup \spec k$ considered
 as a pointed scheme by the obvious $k$-point.

Following Morel,
 our convention for $t$-structures in triangulated categories is \emph{homological}
 -- see in particular Definition \ref{df:htp_t_structure}.

Our monoidal categories, as well as functors between them,
 are always assumed to be symmetric monoidal.
 Similarly, monoids  are always assumed to be commutative.

In diagrams involving shifts or twists of morphisms,
 we do not indicate them in the label of the arrows -- this does not
 lead to any confusion as they are indicated in the source and target
 of the arrows. 

Graded (resp. bigraded) means $\ZZ$-graded (resp. $\ZZ^2$-graded).

\section{The conjecture of Morel} \label{sec:Morel_conj}

\subsection{The homotopy $t$-structure}

\begin{num}
We will denote by $\SH$ the \emph{stable homotopy category over $k$}
 of Morel and Voevodsky (see \cite{Mor1} and \cite{Jar}).
 Objects of $\SH$ will be simply called \emph{spectra}.\footnote{
  They are called $\PP^1$-spectrum in \cite{Mor1} because they
   have to be distinguished from $S^1$-spectra. 
 As we will not use $S^1$-spectra in this work,
  there will be no risk of confusion.}
 It is a triangulated category with a canonical functor:
$$
\smp \rightarrow \SH, X \mapsto \sus X.
$$
It has a closed monoidal structure ; the tensor product 
 denoted by $\wedge$ is characterized by the property:
 $\sus (X \times_k Y)_+=\sus X_+ \wedge \sus Y_+$.
 We denote by $S^0=\sus \spec k_+$ the unit of this tensor product.
By construction of $\SH$, the object $\sus \GG$ 
 is invertible for the tensor product.
For any integer $n \in \ZZ$, we will denote by
$S^{n,n}$ its $n$-th tensor power. Moreover, for any couple 
 $(i,n) \in \ZZ^2$, we put $S^{i,n}=S^{n,n}[i-n]$.
\end{num}

\begin{num}
Consider a spectrum $\E$.
For any smooth $k$-scheme $X$ and any couple $(i,n) \in \ZZ^2$,
 we put $\E^{i,n}(X)=\Hom_{\SH}(\sus X_+,S^{i,n} \wedge \E)$.
 This defines a bigraded cohomology theory on $\sm$.
We let $\spi i n \E$ be the Nisnevich sheaf of abelian groups
 on $\sm$ associated with the presheaf:
$$
X \mapsto \E^{n-i,n}(X).
$$
For a fixed integer $i \in \ZZ$, $\spi i * \E$ will be considered
 as an abelian $\ZZ$-graded sheaf.
 Recall the definition 5.2.1 of \cite{Mor1}:
\end{num}
\begin{df} \label{df:htp_t_structure}
A spectrum $\E$ will be said to be \emph{non negative}
 (resp. \emph{non positive}) if for any $i<0$ (resp. $i>0$)
 $\spi i * \E=0$.
We let $\SH_{\geq 0}$ (resp. $\SH_{\leq 0}$) be the full
 subcategory of $\SH$ made of non negative (resp. nonpositive)
 objects of $\SH$.
\end{df}
As proved in \cite[5.2.3]{Mor1}, this defines a t-structure
 on the triangulated category $\SH$ with \emph{homological conventions}.
 This means the following properties:
\begin{enumerate}
\item The inclusion functor $o_+:\SH_{\geq 0} \rightarrow \SH$
 (resp. $o_-:\SH_{\leq 0} \rightarrow \SH$)
 admits a right adjoint $t_+$ (resp. left adjoint $t_-$).
 For any spectrum $\E$ and any integer $i \in \ZZ$,
  we put:
$$
\begin{array}{ll}
E_{\geq 0}=o_+t_+(\E), \quad & \E_{\leq 0}=o_-t_-(\E), \\
E_{\geq i}=(E[-i])_{\geq 0}[i], & E_{\leq i}=(E[-i])_{\leq 0}[i], \\
E_{>i}=E_{\geq i+1}, & E_{<i}=(E)_{\leq i-1}. 
\end{array}
$$
\item For any spectra $\E,\F$,
\begin{equation}
\Hom_{\SH}(\E_{>0},\F_{\leq 0})=0.
\end{equation}
\item
For any object $\E$ in $\SH$, there is a unique
 distinguished triangle in $\SH$:
\begin{equation}
\E_{\geq 0} \rightarrow \E \rightarrow \E_{<0} \rightarrow \E_{\geq0}[1]
\end{equation}
where the first two maps are given by the adjunctions of point (1).
\end{enumerate}

\begin{rem} \label{rem:generators_positive}
The key point for the previous results is
 the stable $\AA^1$-connectivity theorem of Morel (see  \cite[th. 4.2.10]{Mor1}).
 Recall this theorem also implies that for any smooth $k$-scheme $X$
 and any pair $(i,n) \in \NN \times \ZZ$,
 the spectrum $S^{i+n,n} \wedge \sus X_+$ is non negative.
 Then, according to the previous definition,
  the spectra of this shape constitute an essentially small
 family of generators for the localizing subcategory $\SH_{\geq 0}$ of $\SH$.
\end{rem}

\subsection{Homotopy modules}

\begin{num} \label{num:minus_one}
Given an abelian Nisnevich sheaf $F$ on $\sm$, we denote by 
$F_{-1}(X)$ the kernel of the morphism $F(X \times_k \GG) \rightarrow F(X)$
induced by the unit section of $\GG$.
Following the terminology of Morel, 
 we will say that $F$ is \emph{strictly homotopy invariant}
 if the Nisnevich cohomology presheaf $H^*_\nis(-,F)$ is homotopy invariant.
\end{num}
\begin{df}[Morel] \label{df:htp_modl}
A \emph{homotopy module} is a pair $(F_*,\epsilon_*)$ where $F_*$ is a 
$\ZZ$-graded abelian Nisnevich sheaf on $\sm$ which is strictly homotopy invariant
and $\epsilon_n:F_n \rightarrow (F_{n+1})_{-1}$ is an isomorphism.
A morphism of homotopy modules is
 a homogeneous natural transformation of $\ZZ$-graded sheaves
 compatible with the given isomorphisms.

We denote by $\hmod$ the category of homotopy modules.
\end{df}

\begin{num}
For any spectrum $\E$, $\spi 0 * \E$ has a canonical structure
 of a homotopy module. Moreover, the functor $\E \mapsto \spi 0 * \E$
induces an equivalence of categories
between the heart of $\SH$ for the homotopy $t$-structure
and the category $\hmod$ (see \cite[5.2.6]{Mor1}). 
As in \emph{loc. cit.} we will denote its quasi-inverse by:
\begin{equation} \label{eq:hmod_SH}
H:\hmod \rightarrow \pi_0(\SH).
\end{equation}

Note this result implies that $\hmod$ is a Grothendieck abelian category
 with generators of shape 
\begin{equation} \label{eq:generators_hmod}
\spi 0 * X\{n\}=\spi 0 * {S^{n,n} \wedge \sus X_+}.
\end{equation}
for a smooth $k$-scheme and an integer $n \in \ZZ$.
It admits a monoidal structure defined by
\begin{equation} \label{eq:hmod_tensor}
F_* \otimes G_*:=\pi_0(H(F_*) \wedge H(G_*))_* \simeq (H(F_*) \wedge H(G_*))_{\geq 0}.
\end{equation}
The isomorphism follows from the fact the tensor product
on $\SH$ preserves non negative spectra
 (according to remark \ref{rem:generators_positive}).

Note that $\spi 0 * {S^0}$ is the unit object for this monoidal structure.
Given a homotopy module $F_*=(F_*,\epsilon_*)$
 and an integer $n \in \ZZ$,
 we will denote by $F_*\{n\}$ the homotopy module whose $i$-th graded
 term is $(F_{i+n},\epsilon_{i+n})$.
 Then:
\begin{equation}
F_*\{n\}=\spi 0 * {\GG^{\wedge,n}} \otimes F_*
\end{equation}
so that this notation agrees with that of \eqref{eq:generators_hmod}.
\end{num}

\begin{rem}
Let $F_*$ be a homotopy module.
Then according to the construction of the spectrum $\F:=H(F_*)$,
 for any couple $(i,n) \in \ZZ^2$, we get an isomorphism:
\begin{equation} \label{eq:coh_hmod}
\F^{i,n}(X) \simeq H^{i-n}_\nis(X,F_n),
\end{equation}
natural in $X$.
\end{rem}

\begin{ex} \label{ex:unramified_Milnor&HH}
Let $\HH$ be the spectrum representing
 motivic cohomology.\footnote{
 See \cite{V2} or example \ref{ex:strict_oriented_ring}.}
  Then the homotopy module $\spi 0 * \HH$ is the sheaf
 of unramified Milnor K-theory $\uK_*$ on $\sm$.\footnote{
 This follows from two arguments: the identification of the unstable
  motivic cohomology of bidegree $(n,n)$ of a field $E$ with 
  the $n$-th Milnor ring $K_n^M(E)$; the cancellation theorem of
  Voevodsky to identify unstable motivic cohomology
  with the stable one.}
\end{ex}

\begin{num} \label{num:Hopf}
Define the \emph{Hopf map} in $\SH$ as the morphism $\eta:\GG \rightarrow S^0$
 obtained by applying the tensor product with $S^{-2,-1}$ to the map induced by
$$
\AA^2_k-\{0\} \rightarrow \PP^1_k, (x,y) \mapsto [x,y].
$$
According to \cite[6.2.4]{Mor1},\footnote{See \cite{Mor2} for a proof (stated as Corollary 21
 in the introduction, which can be deduced from the results of section 2.3).}
 there exists a canonical exact sequence in $\hmod$ of shape:
\begin{equation} \label{eq:sec_Morel}
\spi 0 * \GG
 \xrightarrow{\eta_*} \spi 0 * {S^0}
 \rightarrow \uK_* \rightarrow 0.
\end{equation}
\end{num}

\begin{df} \label{df:hmor}
A homotopy module $F_*$ is said to be \emph{orientable}
 if the map induced by $\eta$:
$$
1 \otimes \eta_*:F_*\{1\} \rightarrow F_*
$$
is zero. 
We will denote by $\hmor$ the full subcategory
of $\hmod$ made by the orientable homotopy modules.
\end{df}
Note that $\eta$ is in fact an element of $\pi_0(S^0)_{-1}(k)$.  
Given a homotopy module $F_*$, a smooth scheme $X$ and an integer $n \geq 0$, 
 $H^n_\nis(X,F_*)$ has a canonical structure of
 a graded  $\pi_0(S^0)_*(k)$-module.
Then the following conditions are equivalent:
\begin{enumerate}
 \item[(i)] $F_*$ is orientable.
\item[(ii)] For any smooth scheme $X$ and any integer $n \geq 0$,
 the action of $\eta$ on $H^n_\nis(X,F_*)$ is zero.
\end{enumerate}

\rem \label{rem:mhor&P^1}
\begin{enumerate}
\item Consider a linear embedding
 $i:\PP^1 \rightarrow \PP^2$. 
 Then from \cite[6.2.1]{Mor1},
 the sequence
$$
\sus(\AA^2-\{0\})
 \xrightarrow{\ S^{2,1} \wedge \eta \ } \sus \PP^1
 \xrightarrow{\ i_* \ } \sus \PP^2
$$
is homotopy exact in $\SH$.
In particular, 
 one sees that a homotopy module $F_*$ is orientable if
 $i^*:H^*_\nis(\PP^2,F_*) \rightarrow H^*_\nis(\PP^1,F_*)$
 is split -- actually the reciprocal statement holds
 (see Corollary \ref{cor:carac_mhtpo}).
\item Recall that modulo $2$-torsion,
 if $-1$ is a sum of squares in $k$,
 any homotopy module is orientable (see \cite[6.3.5]{Mor1}).
\end{enumerate}

\subsection{Homotopy modules with transfers}

\begin{num}
Recall Voevodsky has introduced the category $\smc$
 whose objects are the smooth schemes 
 and morphisms are the \emph{finite correspondences}.
 Taking the graph of a morphism of smooth schemes induces a functor
 $\gamma:\sm \rightarrow \smc$.
A Nisnevich sheaf with transfers is an abelian presheaf $F$
 on $\smc$ such that $F \circ \gamma$ is a Nisnevich sheaf.
 Note that the construction $F_{-1}$ of paragraph \ref{num:minus_one}
 applied to a sheaf with transfers $F$ gives in fact a sheaf with transfers,
 still denoted by $F_{-1}$.
 In \cite[1.15]{Deg9}, we have introduced the following definition:
\end{num}
\begin{df} \label{df:hmtr}
A \emph{homotopy module with transfers} is
 a pair $(F_*,\epsilon_*)$ where $F_*$ is a 
 $\ZZ$-graded abelian Nisnevich sheaf with transfers
 which is strictly homotopy invariant
 and $\epsilon_n:F_n \rightarrow (F_{n+1})_{-1}$ is an isomorphism
 of sheaves with transfers.
A morphism of homotopy modules with transfers is
 a homogeneous natural transformation of $\ZZ$-graded sheaves
 with transfers compatible with the given isomorphisms.

We denote by $\hmtr$ the category of homotopy modules with transfers.
\end{df}

\begin{num}
Let $(F_*,\epsilon_*)$ be a homotopy module with transfers.
Obviously, the functor $\gamma_*(F_*)=F_* \circ \gamma$
together with the natural isomorphism $\epsilon_*.\gamma$
is a homotopy module. \\
Recall from \cite[1.3]{Deg9} that one can attach to $F_*$
 a cohomological functor:
$$
\varphi:\DMgm^{op} \rightarrow \ab
$$
from Voevodsky's category of geometric motives
to the category of abelian groups such that: 
$$
\varphi(M(X)(n)[n+i])=H^i_\nis(X,F_{-n}).
$$
According to remark \ref{rem:mhor&P^1},
 the projective bundle theorem in $\DMgm$ implies
 that $\gamma_*(F_*)$ is orientable.
Thus we have obtained a canonical functor:
\begin{equation} \label{eq:hmtr->hmor}
\gamma'_*:\hmtr \rightarrow \hmor.
\end{equation}
One can see that $\gamma'_*$ is faithful.
The main point of this note is the following theorem:
\end{num}
\begin{thm} \label{thm:conj_morel1}
The functor $\gamma'_*$ introduced above is
 an equivalence of categories.
\end{thm}

\begin{num}
This theorem is an equivalent form of the conjecture
 \cite[conj. 3, p. 160]{Mor1} of Morel.
 Indeed, recall the main theorem of \cite{Deg9}
 establishes an equivalence of categories between
 the category of homotopy modules with transfers
 and the category $\modl$ of cycle modules defined by M.~Rost
 in \cite{Ros}.
Indeed, from \cite[3.3]{Deg9}, we get quasi-inverse equivalences of
categories:
\begin{equation}\label{eq:deglise}
\rho:\hmtr \leftrightarrows \modl:A^0.
\end{equation}
In fact, we will prove the previous theorem by proving the following
 equivalent form:
\end{num}
\begin{thm} \label{thm:conj_morel2}
There exists a canonical functor
$$
\tilde \rho:\hmor \rightarrow \modl
$$
and a natural isomorphism of endofunctors of $\hmor$:
$$
\epsilon:1 \rightarrow (\gamma'_* \circ A^0 \circ \tilde \rho).
$$
\end{thm}

\section{Preparations} \label{sec:preparation}

\subsection{The theory of cycle modules}

We shortly recall the theory of cycle modules by M.~Rost
 (see \cite{Ros})
 in a way that will facilitate the proof of our main result.

\subsubsection{Cycle premodules}

\begin{num} \label{num:axioms_pre_modl}
A \emph{function field} 
 will be an extension field $E$ of $k$ of finite transcendental degree.
 A \emph{valued function field} $(E,v)$ will be a function field $E$
 together with a discrete valuation $v:E^\times \rightarrow \ZZ$
 whose ring of integers $\cO_v$ is essentially of finite type over $k$.
 In this latter situation, we will denote by $\kappa(v)$ the residue
 field of $\cO_v$.

Given a function field $E$, we will denote by $K_*^M(E)$ the
 Milnor K-theory of $E$.
 Expanding a remark of Rost (\cite[(1.10)]{Ros}),
 we introduce the additive category $\ecorps$ as follows:
  its objects are the couples $(E,n)$ where $E$
 is a function field and $n \in \ZZ$ an integer;
 the morphisms are defined by the following generators and relations:

\noindent \underline{Generators}~:
\begin{itemize}
\item[\textbf{(D1)}] $\varphi_*:(E,n) \rightarrow (L,n)$
 for an extension field $\varphi:E \rightarrow L$, $n \in \ZZ$.
\item[\textbf{(D2)}] $\varphi^*:(L,n) \rightarrow (E,n)$
 for a finite extension fields $\varphi:E \rightarrow L$,
  $n \in \ZZ$.
\item[\textbf{(D3)}] $\gamma_x:(E,n) \rightarrow (E,n+r)$, for
  any $x \in K_r^M(E)$, $n \in \ZZ$.
\item[\textbf{(D4)}] $\partial_v:(E,n) \rightarrow (\kappa(v),n-1)$,
 for any valued function field $(E,v)$, $n \in \ZZ$.
\end{itemize}

\noindent \underline{Relations}~:
\begin{itemize}
\item[\textbf{(R0)}] For all $x \in K_*^M(E)$, $y \in K_*^M(E)$,
$\gamma_x \circ \gamma_y=\gamma_{x.y}$.
\item[\textbf{(R1a)}] $(\psi \circ \varphi)_*=\psi_* \circ \varphi_*$.
\item[\textbf{(R1b)}] $(\psi \circ \varphi)^*=\varphi^* \circ \psi^*$.
\item[\textbf{(R1c)}] Consider finite extension fields
 $\varphi:K \rightarrow E$, $\psi:K \rightarrow L$
 and put $R=E \otimes_K L$.
 For any point $z \in \spec R$,
 let
  $\bar \varphi_z:L \rightarrow R/z$
  and $\bar \psi_z:E \rightarrow R/z$ be the induced morphisms.
  Under these notations, the relation (R1c) is:

$\psi_*\varphi^*=\sum_{z \in \spec{R}}
\mathrm{lg}\big(R_z\big).(\bar \varphi_z)^*(\bar \psi_z)_*$,

\noindent where $\mathrm{lg}$ stands for the length of an Artinian ring.
\item[\textbf{(R2a)}] For an extension fields $\varphi:E \rightarrow L$,
 $x \in K_*^M(E)$,
 $\varphi_* \circ \gamma_x=\gamma_{\varphi_*(x)} \circ
\varphi_*$.
\item[\textbf{(R2b)}] For a finite extension fields $\varphi:E \rightarrow L$,
 $x \in K_*^M(E)$,
$\varphi^* \circ \gamma_{\varphi_*(x)}=
\gamma_x \circ \varphi^*$.
\item[\textbf{(R2c)}] For a finite extension fields $\varphi:E \rightarrow L$,
 $y \in K_*^M(L)$,
$\varphi^* \circ \gamma_y \circ \varphi_*=
\gamma_{\varphi^*(y)}$.
\item[\textbf{(R3a)}] Let $\varphi:E \rightarrow L$ be a morphism,
 where $(L,v)$ and $(E,w)$ are valued function fields
 such that $v|_{E^\times}=e.w$ for a positive integer $e$. Let
 $\bar\varphi:\kappa(w) \rightarrow \kappa(v)$ be the induced morphism.
 Under these notations, relation (R3a) is:
$\partial_v \circ \varphi_* =e.{\bar \varphi}_* \circ \partial_w$.
\item[\textbf{(R3b)}] Let $\varphi:E \rightarrow L$ be a finite extension fields
 where $(E,v)$ is a valued function field. For any valuation $w$ on $L$,
 we let $\bar \varphi_w:\kappa(v) \rightarrow \kappa(w)$ be the induced morphism.
 Then relation (R3b) is:
$\partial_v \circ \varphi^*=\sum_{w/v} \bar \varphi_w^* \circ \partial_w$.
\item[\textbf{(R3c)}]For any $\varphi:E \rightarrow L$, 
 any valuation $v$ on $L$, trivial on $E^\times$: $\partial_v \circ \varphi_*=0$.
\item[\textbf{(R3d)}] For a valued function field $(E,v)$, a prime $\pi$ of $v$~:
 $\partial_v \circ \gamma_{\{-\pi\}} \circ \varphi_* ={\bar \varphi}_*$.
\item[\textbf{(R3e)}] For a valued field $(E,v)$ and a unit $u \in E^\times$~:
$\partial_v \circ \gamma_{\{u\}}
=-\gamma_{\{\bar u\}} \circ \partial_v$.
\end{itemize}
Recall the following definition -- \cite[(1.1)]{Ros}:
\end{num}
\begin{df}
A cycle premodule is an additive covariant functor
$M:\ecorps \rightarrow \ab$.
We denote by $\pmodl$ the category of such functors
 with natural transformations as morphisms. 
\end{df}

\subsubsection{Cycle modules}

\begin{num}
Consider a cycle premodule $M$,
 a scheme $X$ essentially of finite type over $k$,
 and an integer $p \in \ZZ$. 
 According to \emph{loc. cit.}, \textsection 5, we set:
\begin{equation}
C^p(X;M)=\bigoplus_{x \in X^{p}} M(\kappa_x)
\end{equation}
where $\kappa_x$ denotes the residue field of $x$ in $X$.
This is a graded abelian group and we put: $C^p(X;M)_n
 =\bigoplus_{x \in X^{p}} M_{n-p}(\kappa_x)$.\footnote{
 It is denoted by $C^p(X;M,n)$ in \emph{loc. cit.}}

Consider a couple $(x,y) \in X^{(p)} \times X^{(p+1)}$.
Assume that $y$ is a specialization of $x$.
Let $Z$ be the reduced closure of $x$ in $X$
 and $\tilde Z \xrightarrow f Z$ be its normalization. 
Each point $t \in f^{-1}(y)$ corresponds to a
discrete valuation $v_t$ on $\kappa_x$ with residue 
field $\kappa_t$.
We denote by $\varphi_t:\kappa_y \rightarrow \kappa_t$ the 
morphism induced by $f$. Then,
 we define according to \cite[(2.1.0)]{Ros} the following 
 homogeneous morphism of graded abelian groups:
\begin{equation} \label{eq:residues_pt_modl}
\partial_y^x=
\sum_{t \in f^{-1}(y)}
 \varphi_t^* \circ \partial_{v_t}:
 C^p(X;M)_n \rightarrow C^{p+1}(X;M)_n
\end{equation}
If $y$ is not a specialization of $x$,
  we put: $\partial_y^x=0$.
\end{num}
\begin{df}
Consider the hypothesis and notations above.
We introduce the following property of the cycle premodule $M$:
\begin{itemize}
\item[\pFD X] For any $x \in X$ and any $\rho \in M(\kappa_x)$,
 $\partial_y^x(\rho)=0$ for all but finitely many $y \in X$.
\end{itemize}
Assume \pFD X is satisfied. Then for any integer $p$,
 we define according to \cite[(3.2)]{Ros} the following morphism:
\begin{equation} \label{eq:diff_modl}
d^p_X=\sum_{(x,y) \in X^{(p)} \times X^{(p+1)}} \partial_y^x:
 C^p(X;M)_n \rightarrow C^{p+1}(X;M)_n,
\end{equation}
and introduce the following further property of $M$:
\begin{itemize}
\item[\pC X] For any integer $p \geq 0$, $d^{p+1}_X \circ d^p_X=0$.
\end{itemize}
\end{df}
\begin{num}
Thus, under \pFD X and \pC X,
 the map $d^*_X$ is a differential on $C^*(X;M)$
 and we get the complex of cycles in $X$ with coefficients in $M$
 defined by Rost in \emph{loc. cit.} In fact, the complex $C^*(X;M)$
 is graded according to \eqref{eq:diff_modl}.
\end{num}

In the conditions of the above definition,
 the following properties are clear from the definition:
\begin{itemize}
\item Let $Y$ be a closed subscheme (resp. open subscheme, localized scheme) of $X$.
Then \pFD X implies \pFD Y and \pC X implies \pC Y.
\item Let $X=\cup_{i \in I} U_i$ be an open cover of $X$.
Then \pFD X is equivalent to \pFD {U_i} for any $i \in I$.
\end{itemize}
As any affine algebraic $k$-scheme $X$ can be embedded into $\AA^n_k$
 for $n$ sufficiently large,
 we deduce easily from these facts the following lemma:
\begin{lm} \label{lm:reduce_modl}
Let $M$ be a cycle premodule.
Then the following conditions are equivalent:
\begin{enumerate}
\item[(i)] $M$ satisfies \pFD X for any $X$.
\item[(ii)] $M$ satisfies \pFD{\AA^n} for any integer $n \geq 0$.
\end{enumerate}
Assume these equivalent conditions are satisfied. Then,
 the following conditions are equivalent:
\begin{enumerate}
\item[(iii)] $M$ satisfies \pC X for any $X$.
\item[(iv)] $M$ satisfies \pC{\AA^n} for any integer $n \geq 0$.
\end{enumerate}
\end{lm}
The following definition is equivalent to
 that of \cite[(2.1)]{Ros} (see in particular \cite[(3.3)]{Ros}):
\begin{df}
A cycle premodule $M$ which satisfies the equivalent conditions
 (i)-(iv) of the previous lemma is called a cycle module.
\end{df}

\begin{num} \label{num:df_Chow_gp}
Given a cycle module $M$ and a scheme $X$ essentially of finite type
 over $k$, we denote by $A^p(X;M)$ the $p$-th cohomology group
 of the complex $C^*(X;M)$.  This group is graded according
 to the graduation of $C^*(X;M)$.

Recall that according to one of the main construction of Rost,
 $A^*(X;M)$ is contravariant with respect to morphisms of smooth schemes.
 Moreover, according to \cite{Deg5}, the presheaf
$$
X \mapsto A^0(X;M)
$$
has a canonical structure of a homotopy module with transfers
 (def. \ref{df:hmtr}) which actually defines the functor
 $A^0$ of \eqref{eq:deglise}.
\end{num}

\subsection{Modules and ring spectra}

\begin{num} \label{num:SS_model}
Recall that a \emph{ring spectrum} is
 a (commutative) monoid $\R$ of
  the monoidal category $\SH$.

A module over the ring spectrum $\R$ is a $\R$-module in $\SH$
 in the classical sense: a spectrum $\E$ equipped
 with a multiplication map $\gamma_\E:\R \wedge \E \rightarrow \E$
 satisfying the usual identities -- see \cite{Mcl}.\footnote{As $\R$
 is commutative, we will not distinguished
  the left and right $\R$-modules.}

Given two $\R$-modules $\E$ and $\F$,
 a morphism of $\R$-modules is a morphism $f:\E \rightarrow \F$ in $\SH$
 such that the following diagram is commutative:
$$
\xymatrix@R=12pt@C=28pt{
\R \wedge \E\ar^{1_\R \wedge f}[r]\ar_{\gamma_\E}[d]
 & \R \wedge \F\ar^{\gamma_\F}[d] \\
\E\ar^f[r] & \F.
}
$$
This defines an additive category denoted by $\wmod \R$.

Given any spectrum $\E$, $\R \wedge \E$ has an obvious structure
 of a $\R$-module. This defines a functor which is obviously left
 adjoint to the inclusion functor. Thus we get an adjunction of categories:
\begin{equation} \label{eq:SH_wmod}
L_\R^w:\SH \leftrightarrows \wmod \R:\cO_\R^w.
\end{equation}
\end{num}

\rem The category $\wmod \R$ has a poor structure.
For example, it is not possible in general to define a tensor product
 or a triangulated structure on $\wmod \R$. 
 This motivates the definitions which follow.

\begin{num} \label{num:strict_ring_sp}
Recall also from \cite{Jar} that $\SH$ is the homotopy category
 of a monoidal model category: the category of
 symmetric spectra denoted by $\Sp$.
A \emph{strict ring spectrum} $\R$ is a commutative monoid object
 in the category $\Sp$. Given such an object we can form
 the category $\spmod \R$ of $\R$-modules with respect to 
 the monoidal category $\Sp$ and consider the natural
 adjunction
\begin{equation} \label{eq:free-Rmod_Sp}
\Sp \leftrightarrows \spmod \R
\end{equation}
where the right adjoint is the forgetful functor.
An object of $\spmod \R$ will be called a \emph{strict module
over $\R$}.

The monoidal model category $\Sp$ satisfies the monoid axiom
 of Schwede and Shipley: this implies that the category of
 strict $\R$-modules admits a (symmetric) monoidal model structure
 such that the right adjoint of \eqref{eq:free-Rmod_Sp} preserves
 and detects fibrations and weak equivalences
 (cf \cite[4.1]{SS}).
\end{num}
\begin{df}
We denote by $\smod \R$ the homotopy category
 associated with the model category $\spmod \R$ 
 described above.
\end{df}
\begin{num} \label{num:adj_smod}
Note that $\smod \R$ is a monoidal triangulated category.
The Quillen adjunction \eqref{eq:free-Rmod_Sp} induces a canonical
 adjunction:
\begin{equation} \label{eq:free-Rmod_SH}
L_\R:\SH \leftrightarrows \smod \R:\cO_\R.
\end{equation}
By construction,
 the functor $L_\R$ is triangulated and monoidal.
Given a smooth scheme $X$ and a couple $(i,n) \in \ZZ^2$,
 we will put:
\begin{equation} \label{free_Rmod}
\uR(X)(n)[i]:=L_\R(S^{i,n} \wedge \sus X_+).
\end{equation}
Note that the $\R$-modules of the above shape are compacts
 and form a family of generators for the triangulated category $\smod \R$.

The functor $\cO_\R$ is triangulated and
 conservative. Because it is the right adjoint of a monoidal functor,
 it is weakly monoidal.
In other words, for any strict $\R$-modules $M$ and $N$,
we get a canonical map in $\SH$:
$$
\cO_\R(M \otimes_\R N)
 \rightarrow \cO_\R(M) \wedge \cO_\R(N).
$$
This implies that $\cO_\R$ maps into the subcategory
 $\wmod \R$ giving a functor
\begin{equation} \label{eq:strict_to_weak_mod}
\cO_\R':\smod \R \rightarrow \wmod \R.
\end{equation}
For any spectrum $\E$, $\cO_\R L_\R(\E)=\R \wedge \E$.
Thus the following diagram commutes:
\begin{equation} \label{eq:strict_to_weak_mod_ppty}
\begin{split}
\xymatrix@=14pt{
& \SH\ar^{L^w_\R}[rd]\ar_{L_\R}[ld] & \\
\smod \R\ar|{\cO'_\R}[rr] && \wmod \R.
}
\end{split}
\end{equation}
\end{num}

\begin{ex} \label{ex:strict_oriented_ring}
The spectrum $\HH$ 
 representing motivic
 cohomology has a canonical structure of a strict ring spectrum.
This follows from the adjunction of triangulated categories
\begin{equation} \label{eq:SH&DM}
\gamma^*:\SH \leftrightarrows \DM:\gamma_*
\end{equation}
where $\DM$ is the category of stable motivic complexes over $k$.
In fact, by the very definition, 
$$
\HH=\gamma_*(\mathbbm 1)
$$
where $\mathbbm 1$ is the unit object for the monoidal structure on $\DM$.
To prove that $\HH$ is a strict ring spectrum,
 the argument is that $\gamma^*$ is induced by
 a monoidal left Quillen functor between the underlying monoidal
 model categories (see \cite{CD3} for details).

Note also that any object of $\DM$ admits a canonical
 structure of a strict $\HH$-module. 
 In fact, 
 the previous adjunction induces a canonical adjunction
\begin{equation}\label{eq:Hmod&DM}
\tilde \gamma^*:\smod \HH \leftrightarrows \DM:\tilde \gamma_*
\end{equation}
(see again \cite{CD3}).
\end{ex}

\subsection{On the Gysin triangle and morphism}
\label{sec:recall_Gysin}

\begin{num} \label{Deg8applies}
The proof of theorem \ref{thm:generic_smod}
 requires some of the results of \cite{Deg8}.
 We recall them to the reader to make the proof more intelligible.

Recall that for any smooth scheme $X$,
 we have\footnote{Note that,
 as in example \ref{ex:unramified_Milnor&HH},
 one uses the cancellation theorem of Voevodsky to get this isomorphism.}
  an isomorphism of abelian groups:
\begin{equation} \label{eq:Pic&motivic_coh}
c_1:\pic(X) \rightarrow \HH^{2,1}(X).
\end{equation}
Considering the left adjoint of \eqref{eq:free-Rmod_SH},
 we get a monoidal functor:
\begin{equation}\label{eq:sh->Hmod}
\SH \rightarrow \smod \HH
\end{equation}
and this functor shows the category $\smod \HH$ satisfies the axioms
 of paragraph 2.1 in \emph{loc. cit.}
  -- see example 2.12 of \emph{loc. cit.}
\end{num}

\begin{rem}\label{rem:comput_r_times}
Note that,
 because \eqref{eq:Pic&motivic_coh} is a morphism of group,
 the formal group law attached to the category $\smod \HH$
 in \cite[3.7]{Deg8} is just the additive formal group
 law $F(x,y)=x+y$.
In particular, the power series in the indeterminate $x$
$$
[2]_F=F(x,x), [3]_F=F(x,F(x,x)), \hdots
$$ 
are simply given by the formula:
$[r]_F=r.x$ for any integer $r$.
\end{rem}

\begin{num}
\textit{The Gysin triangle}: For any smooth schemes $X$, $Z$ and
 any closed immersion $i:Z \rightarrow X$ of pure codimension $c$,
 with complementary open immersion $j$,
 there exists a canonical distinguished triangle (see \cite[def. 4.6]{Deg8}):
\begin{equation} \label{eq:Gysin_tri_smod}
\uHH(Z-X) \xrightarrow{j_*} \uHH(X) \xrightarrow{i^*}
 \uHH(Z)(c)[2c] \xrightarrow{\partial_{X,Z}} \uHH(X-Z)[1].
\end{equation}
The map $i^*$ (resp. $\partial_{X,Z}$)
 is called the Gysin (resp. residue) morphism associated with $i$.
Consider moreover a commutative square of smooth schemes
\begin{equation} \label{eq:cartesian_sq_lambda}
\begin{split}
\xymatrix@=20pt{
T\ar^k[r]\ar_q[d] & Y\ar^p[d] \\
Z \ar^i[r] & X
}
\end{split}
\end{equation}
such that $i$ and $k$ are closed immersions of pure codimension $c$
 and $T$ is the reduced scheme associated with $Z \times_X Y$.
 Let $h:(Y-T) \rightarrow (X-Z)$ be the morphism induced by $p$.
Then the following formulas hold:
\begin{itemize}
\item[(G1a)] Assume $T=Z \times_X Y$.
Then, according to \cite[prop. 4.10]{Deg8}, the following diagram is commutative:
\begin{equation*}
\xymatrix@R=20pt@C=30pt{
\uHH(Y)\ar^-{k^*}[r]\ar_{p_*}[d]
 & \uHH(T)(c)[2c]\ar^-{\partial_{Y,T}}[r]\ar_{q_*}[d]
  & \uHH(Y-T)[1]\ar^{h_*}[d] \\
\uHH(X)\ar^-{i^*}[r]
 & \uHH(Z)(c)[2c]\ar^-{\partial_{X,Z}}[r] & \uHH(X-Z)[1].
}
\end{equation*}
\item[(G1b)] Assume $Z \times_X Y$ is irreducible
 and let $e$ be its geometric multiplicity.
 Then according to \cite[4.26]{Deg8} combined with remark \ref{rem:comput_r_times},
  the following diagram is commutative:
\begin{equation*}
\xymatrix@R=20pt@C=30pt{
\uHH(Y)\ar^-{k^*}[r]\ar_{p_*}[d]
 & \uHH(T)(c)[2c]\ar^-{\partial_{Y,T}}[r]\ar_{e.q_*}[d]
  & \uHH(Y-T)[1]\ar^{h_*}[d] \\
\uHH(X)\ar^-{i^*}[r]
 & \uHH(Z)(c)[2c]\ar^-{\partial_{X,Z}}[r] & \uHH(X-Z)[1].
}
\end{equation*}
\end{itemize}
\end{num}

\begin{rem} \label{rem:fdl_class}
The Gysin morphism $i^*$ in the triangle \eqref{eq:Gysin_tri_smod}
 induces a pushout in motivic cohomology:
$$
i_*:\HH^{n,m}(Z) \rightarrow \HH^{n+2c,m+c}(X).
$$
Considering the unit element of the graded algebra $\HH^{*,*}(Z)$,
 we define the fundamental class associated with $i$ 
 as the element $\eta_X(Z):=i_*(1) \in \HH^{2c,c}(X)$.

Assume $c=1$.
 Then $Z$ defines an effective Cartier divisor of $X$
 which corresponds uniquely to a line bundle $L$ over $X$
 together with a section $s:X \rightarrow L$.
 Moreover, $s$ is transversal to the zero section $s_0$
 and the following square is cartesian:
$$
\xymatrix@=10pt{
Z\ar^-i[r]\ar_i[d] & X\ar^s[d] \\
X\ar^{s_0}[r] & L
}
$$
Then according to \cite[4.21]{Deg8},
 $\eta_X(Z)=c_1(L)$ with the notation of \eqref{eq:Pic&motivic_coh}.\footnote{
 One can prove more generally that, through the isomorphism
 $\HH^{2c,c}(X) \simeq CH^c(X)$ where the left hand side denotes the Chow group
 of codimension $c$ cycles, $i_*$ agrees with the usual pushout on Chow groups
 (see \cite[Prop. 3.11]{Deg9}).
 Thus $\eta_X(Z)$ simply corresponds to the cycle of $Z$ in $X$.}
\end{rem}

\begin{num} \label{num:prelim_lm:add_Gysin}
Consider a smooth scheme $X$ such that
 $X=X_1 \sqcup X_2$ and let $x_1:X_1 \rightarrow X$
 be the canonical open and closed immersion.
 By additivity, $\uHH(X)=\uHH(X_1) \oplus \uHH(X_2)$.
 Moreover, $x_{1*}$ is the obvious split inclusion
 which corresponds to this decomposition.
 Correspondingly, it follows from (G1a) that
  $x_1^*$ is the obvious split epimorphism.
 One can complete \cite{Deg8} by the following lemma
  which describes the additivity properties of the Gysin triangle:
\end{num}
\begin{lm} \label{lm:add_Gysin}
Consider smooth schemes $X$, $Z$ and
 a closed immersion $\nu:Z \rightarrow X$ of pure codimension $n$.

Consider the canonical decompositions $Z=\sqcup_{i \in I} Z_i$
 and $X=\sqcup_{j \in J} X_j$ into connected components
 and put $\hat Z_j=Z \times_X X_j$.
 We also consider the obvious inclusions:
$$
z_i:Z_i \rightarrow Z, 
 x_j:X_j \rightarrow X, 
 u_j:(X_j-\hat Z_j) \rightarrow (X-Z)
$$
and the morphisms $\nu_{ji}$, $\partial_{ij}$ uniquely defined by the following commutative diagram:
$$
\xymatrix@C=55pt{
\uHH(X)\ar^{\nu^*}[r]
 & \uHH(Z)(n)[2n]\ar^{\partial_{X,Z}}[r]
 & \uHH(X-Z)[1] \\
\oplus_{j \in J} \uHH(X_j)\ar_-{(\nu_{ji})_{j \in J,i \in I}}[r]
  \ar_\sim^{\sum_j x_{j*}}[u]
 & \oplus_{i \in I} \uHH(Z_i)(n)[2n]\ar_-{(\partial_{ij})_{i \in I, j \in J}}[r]
 \ar^\sim_{\sum_i z_{i*}}[u]
 & \oplus_{j \in J} \uHH(X_j-\hat Z_j)[1].
  \ar^\sim_{\sum_j u_{j*}}[u]. \\
}
$$
Then, for any couple $(i,j) \in I \times J$:
\begin{enumerate}
\item if $Z_i \subset X_j$, we let $\nu_i^j:Z_i \rightarrow X_j$
 be the obvious inclusion and we put: $Z'_i=\hat Z_j-Z_i$. \\
 Then $\nu_{ji}=\big(\nu_i^j\big)^*$
 and $\partial_{ij}=\partial_{X_j-Z'_i,Z_i}$.
\item Otherwise,
$\nu_{ji}=0$  and $\partial_{ij}=0$.
\end{enumerate}
\end{lm}
\begin{proof}
According to the preamble \ref{num:prelim_lm:add_Gysin},
 we get:
$\nu_{ji}=z_i^* \nu^* x_{j*}$, $\partial_{i,j}=u_j^* \partial_{X,Z} z_{i*}$.

In the respective case (1) and (2), we consider the
 following cartesian squares:
\begin{equation*} \label{eq:pf:prop:add_Gysin}
\begin{split}
\begin{array}{ll}
\text{(1) If } Z_i \subset X_j, &
\text{(2) otherwise,} \\
\xymatrix@=30pt{
Z_i\ar^{\nu_i^j}[r]\ar@{=}[d]\ar@{}|{\Delta_1}[rd]
 & X_j\ar|{\phantom{(}x_j\phantom{)}}[d]\ar@{}|{\Delta_2}[rd]
 & \hat Z_j\ar|{\phantom{\big(}}[d]\ar@{}|{\nu_i^j}[d]\ar_-{\hat z_j}[l]\ar@{}|{\Delta_3}[rd]
 & Z_i\ar_-{}[l]\ar@{=}[d] \\
Z_i\ar_{\nu_i}[r] & X
 & Z\ar^\nu[l] & Z_i\ar^{z_i}[l]
} \quad
&
\xymatrix@=30pt{
\emptyset\ar[r]\ar[d]\ar@{}|{\Delta_1}[rd]
 & X_j\ar|/-1pt/{\phantom{(}}[d]\ar@{}|{x_j}[d]\ar@{}|{\Delta_2}[rd]
 & \hat Z_j\ar|{\phantom{\big(}}[d]\ar@{}|{\nu_i^j}[d]\ar_-{\hat z_j}[l]\ar@{}|{\Delta_3}[rd]
 & \emptyset\ar[l]\ar[d] \\
Z_i\ar_{\nu_i}[r] & X
 & Z\ar^\nu[l] & Z_i\ar^{z_i}[l]
}
\end{array}
\end{split}
\end{equation*}
Then the result follows from the following computations:
\begin{align*}
z_i^* \nu^* x_{j*}
 \stackrel{(a)}=\nu_i^* x_{j*}
 \stackrel{(b)}=
&\begin{cases}
(\nu_i^j)^*  & \text{if } Z_i \subset X_j, \\
0 & \text{otherwise}.
\end{cases} \\
u_j^* \partial_{X,Z} z_{i*}
 \stackrel{(c)}=\partial_{X_j,\hat Z_j} \hat z_j^* z_{i*}
 \stackrel{(d)}=
&\begin{cases}
\partial_{X_j,\hat Z_j} (z_i^j)_*\stackrel{(e)}=\partial_{X_j-Z'_i,Z_i}
  & \text{if } Z_i \subset X_j, \\
0 & \text{otherwise}.
\end{cases}
\end{align*}
We give the following justifications for each equality:
(a) use (G2a) ($\nu_i=\nu \circ z_i$), 
(b) use (G1a) applied to $\Delta_1$,
(c) use (G2b) applied to $\Delta_2$,
(d) use (G1a) applied to $\Delta_3$,
(e) use (G1a).
\end{proof}

\begin{num} \label{num_Gysin_morphism_ppty}
\textit{The Gysin morphism}: For any smooth schemes
 $X$ and $Y$ and any projective equidimensional morphism
  $f:Y \rightarrow X$ of dimension $n$,
 there exists (see \cite[def. 5.12]{Deg8})
 a canonical morphism in $\smod \HH$ of the form
\begin{equation} \label{eq:Gysin_morph}
f^*:\uHH(X)(n)[2n] \rightarrow \uHH(Y)
\end{equation}
which coincides with the morphism of \eqref{eq:Gysin_tri_smod}
 in the case where $f$ is a closed immersion. It satisfies the following properties:
\begin{itemize}
\item[(G2a)] Whenever it makes sense, we get (see \cite[prop. 5.14]{Deg8}):
$f^*g^*=(gf)^*.$
\item[(G2b)] Consider a cartesian square of smooth schemes of shape \eqref{eq:cartesian_sq_lambda}
 such that $i$ and $k$ (resp. $p$ and $q$)
 are closed immersion (resp. projective equidimensional morphisms).
 Assume $i$, $k$, $f$, $g$ have constant relative codimension
 $n$, $m$, $s$, $t$ respectively. Put $d=n+t=m+s$.
 Let $h:(Y-T) \rightarrow (X-Z)$ be the morphism induced by $f$.
 Then the following square is commutative
   (see \cite[prop. 5.15]{Deg8}\footnote{There is a misprint in \emph{loc. cit.}: 
   one should read $n+t=m+s$, $d=n+t$.}):
\begin{equation*} \label{eq:Gysin&residue}
\xymatrix@R=20pt@C=35pt{
\uHH(Z)(n)[2n]\ar^-{\partial_{X,Z}}[r]\ar_{g^*}[d] & \uHH(X-Z)[1]\ar^{h^*}[d] \\
\uHH(T)(d)[2d]\ar^-{\partial_{Y,T}}[r] & \uHH(Y-T)(s)[2s+1].
}
\end{equation*}
\item[(G2c)] For any cartesian square of smooth schemes
\begin{equation} \label{eq:comm_sq_lambda_bis}
\begin{split}
\xymatrix@=20pt{
T\ar^q[r]\ar_g[d]\ar@{}|\Delta[rd] & Y\ar^f[d] \\
Z\ar^p[r] & X
}
\end{split}
\end{equation}
such that $p$ and $q$ are projective equidimensional of the same dimension,
 we get (see \cite[5.17(i)]{Deg8}): 
\begin{equation*}
p^*f_*=g_*q^*.
\end{equation*}
\item[(G2d)] Consider a commutative square of smooth schemes
 of shape \eqref{eq:comm_sq_lambda_bis} such that:
\begin{itemize}
\item $p$, $q$, $f$ and $g$ are finite equidimensional morphisms.
\item $T$ is equal to the reduced scheme associated with $T'=Z \times_X Y$.
\end{itemize}
Let $(T'_i)_{i \in I}$ be the connected components of $T'$.
For any index $i \in I$, 
 we let $q_i:T'_i \rightarrow X'$ (resp.
 $g_i:T_i' \rightarrow Y$) be the morphism
 induced by $q$ (resp. $g$)
 and we denote by $r_i$ be the geometric multiplicity
 of $T'_i$.
Then according to \cite[prop. 5.22]{Deg8}
 combined with remark \ref{rem:comput_r_times},
 we get the following formula:
\begin{equation*}
p^*f_*=\sum_{i \in I} r_i.g_{i*}q_i^*.
\end{equation*}
\end{itemize}
\end{num}

\begin{rem} \label{rem:G2b_plus}
At one point\footnote{The proof of Proposition \ref{compute_differentials}.},
 we will need property (G2b) when we only require that
 the square \eqref{eq:cartesian_sq_lambda} is topologically cartesian
 (\emph{i.e.} $T=(Y \times_X Z)_{red}$).
In fact, the proof given in \cite{Deg8} requires only this
 assumption: one proves Proposition 5.15 by reducing to Theorem 4.32
 (the case where $p$ and $q$ are closed immersions),
  and the proof uses only the fact the square is topologically cartesian
  (in the construction of the isomorphism denoted by $\mathfrak p_{(X;Y,Y')}$).
\end{rem}

\begin{num} \label{num:products}
\textit{Products}: Let $X$ be a smooth scheme
 and $\delta:X \rightarrow X \times X$ the diagonal embedding.
 For any morphisms $a:\uHH(X) \rightarrow \E$ and $b:\uHH(X) \rightarrow \F$,
 we denote by $a \ecuppx X b$ the composite morphism:
\begin{equation} \label{eq:product}
\uHH(X) \xrightarrow{\delta_*} \uHH(X) \otimes_\HH \uHH(X)
 \xrightarrow{a \otimes_\HH b} \E \otimes_\HH \F.
\end{equation}
\begin{itemize}
\item[(G3a)] For any morphism $f:Y \rightarrow X$ of smooth schemes
 and any couple $(a,b)$ of morphisms in $\smod R$ with source $\uHH(Y)$,
 we get (obviously):
$$
(a \ecuppx X b) \circ f_*=(a \circ f_*) \ecuppx Y (b \circ f_*)
$$
\item[(G3b)] For any projective morphism $f:Y \rightarrow X$
 of smooth schemes and any morphism $a$ (resp. $b$) in $\smod \HH$
 with source $\uHH(X)$ (resp. $\uHH(Y)$), we get
 (see \cite[5.18]{Deg8}):
\begin{align*}
\lbrack a \ecuppx X (b \circ f_*)\rbrack \circ f^*
=(a \circ f^*) \ecuppx Y b. \\
\lbrack (b \circ f_*) \ecuppx X a \rbrack \circ f^*
=b \ecuppx Y (a \circ f^*).
\end{align*}
\end{itemize}
\end{num}

\begin{rem} \label{rem:comput_fdl_class}
Consider a smooth scheme $X$ and
 a vector bundle $E/X$ of rank $n$. Let $P=\PP(E)$ be the associated
 projective bundle with projection $p:P \rightarrow X$
 and $\lambda$ the canonical line bundle such that
 $\lambda \subset p^{-1}(E)$.
  Then, for any integer $r \geq 0$,
 the $r$-th power of the class $c=c_1(\lambda^\vee)$ (see \eqref{eq:Pic&motivic_coh})
 in motivic cohomology corresponds to a morphism in $\smod \HH$
 which we denote by:
 $c^r:\uHH(P) \rightarrow \uHH(r)[2r]$.
The projective bundle theorem (\cite[3.2]{Deg8}) says that the map
$$
\sum_{0 \leq r \leq n} p_* \ecuppx P c^r:\uHH(P)
 \rightarrow \oplus_{0 \leq r \leq n} \uHH(X)(r)[2r].
$$
is an isomorphism. Thus we get a canonical map
 $\mathfrak l_n(P):\uHH(X)(n)[2n] \rightarrow \uHH(P)$.

Given a finite epimorphism $f:Y \rightarrow X$
 which admits a factorization as
$$
Y \xrightarrow i P \xrightarrow p X
$$
where $i$ is a closed immersion and $P/X$ is the projective bundle considered above,
we recall that the Gysin morphism $f^*$ of \eqref{eq:Gysin_morph} is defined 
as the composite map:
$$
\uHH(X)(n)[2n] \xrightarrow{\mathfrak l_n(P)} \uHH(P)
 \xrightarrow{i^*} \uHH(Y)(n)[2n]
$$
after taking the tensor product with $\HH(-n)[-2n]$.

Assume $Y$ (thus $X$) is connected and $P/X$ has relative dimension $1$.
Consider the map
 $f_*:\HH^{n,m}(Y) \rightarrow \HH^{n,m}(X)$ 
 induced by $f^*$. 
 According to the above description,
  we get a factorization of $f_*$ as:
$$
\HH^{n,m}(Y) \xrightarrow{i_*} \HH^{n+2,m+1}(P)
 \xrightarrow{p_*} \HH^{n,m}(X).
$$
where $i_*$ corresponds to the morphism of 
 remark \ref{rem:fdl_class}.
 To describe the second map, recall that any class
 $\alpha \in \HH^{n+2,m+1}(P)$ can be written uniquely
 as $\alpha=p^*(\alpha_0)+p^*(\alpha_1).c_1(\lambda^\vee)$:
 then $p_*(\alpha)=\alpha_1$.
Thus, we now deduce from this description the following \emph{trace formula}:
\begin{equation} \label{eq:trace_formula}
f_*(1)=d
\end{equation}
where $d$ is the generic degree of $f$.
In fact, according to remark \ref{rem:fdl_class}
 and its notations, $i_*(1)=c_1(L)$.
 Thus the formula follows from the equality
 $c_1(L)=d.c_1(\lambda^\vee)$
 in $\pic(P)$ modulo $\pic(X)$.
\end{rem}

\begin{num}
Consider a regular invertible function $x:X \rightarrow \GG$
 on a smooth scheme $X$.
According to the canonical decomposition $\uHH(\GG)=\uHH \oplus \uHH(1)[1]$,
 it induces a morphism:
$x':\uHH(X) \rightarrow \uHH(1)[1]$.

Using the product \eqref{eq:product}, we then deduce the following morphism:
\begin{equation} \label{eq:action_units}
\gamma_x=1_{X*} \ecuppx X x':\uHH(X) \rightarrow \uHH(X)(1)[1].
\end{equation}
If $\nu_x:X \rightarrow X \times \GG$ denotes the graph of $x$,
 then $\gamma_x$ is also equal to the following composite map:
\begin{equation} \label{eq:action_units_bis}
\uHH(X) \xrightarrow{\nu_{X*}} \uHH(X \times \GG)=\uHH(X) \otimes_\HH \uHH(\GG)
 \rightarrow \uHH(X)(1)[1].
\end{equation}
We will need the following properties of this particular kind of products:
\end{num}
\begin{prop} \label{prop:residues_units}
Let $X$ be a smooth scheme
 and $i:Z \rightarrow X$ be the immersion of a smooth divisor.
 Put $U=X-Z$ and let $j:U \rightarrow X$ be the canonical open immersion.
\begin{enumerate}
\item Let $x:X \rightarrow \GG$ be a regular invertible function,
 $\bar u$ (resp. $u$) its restriction to $Z$ (resp. $U$).
 Then the following diagram is commutative:
$$
\xymatrix@R=20pt@C=50pt{
\uHH(Z)(1)[1]\ar^-{\partial_{X,Z}}[r]\ar_{\gamma_{\bar u} \otimes_\HH Id}[d]
 & \uHH(U)\ar^{\gamma_u}[d] \\
\uHH(Z)(2)[2]\ar^-{-\partial_{X,Z} \otimes_\HH Id}[r] & \uHH(U)(1)[1]
}
$$
\item Consider a regular function $\pi:X \rightarrow \AA^1$
 such that $Z=\pi^{-1}(0)$. Write again $\pi:U \rightarrow \GG$
 for the obvious restriction of $\pi$. Then the following diagram
 is commutative:
$$
\xymatrix@R=14pt@C=36pt{
\uHH(Z)(1)[1]\ar^-{\partial_{X,Z}}[r]\ar_{i_*}[rd]
 & \uHH(U)\ar^-{\gamma_\pi}[r] 
 & \uHH(U)(1)[1]\ar^{j_*}[ld] \\
& \uHH(Z)(1)[1] &
}
$$
\end{enumerate}
\end{prop}
\begin{proof}
\noindent (1) Let $\nu_x$, $\nu_u$, $\nu_{\bar u}$ be the respective graphs
 of $x$, $u$ and $\bar u$.
Applying property (G1a), we get a commutative square:
$$
\xymatrix@R=16pt@C=64pt{
\uHH(Z)(1)[1]\ar^-{\partial_{X,Z}}[r]\ar_{\nu_{\bar u*} \otimes_\HH Id}[d]
 & \uHH(U)\ar^{\nu_{u*}}[d] \\
\uHH(Z \times \GG)(1)[1]\ar^-{\partial_{X \times \GG,Z \times \GG}}[r]
 & \uHH(U \times \GG)
}
$$
According to \cite[4.12]{Deg8}, we get a commutative diagram:
$$
\xymatrix@R=16pt@C=40pt{
\uHH(Z \times \GG)(1)[1]\ar^-{\partial_{X \times \GG,Z \times \GG}}[rr]
 \ar@{=}[d]
 && \uHH(U \times \GG)\ar@{=}[d] \\
\uHH(Z) \otimes_\HH \uHH(\GG)(1)[1]\ar^\epsilon_\sim[r]
 & \uHH(Z)(1)[1] \otimes_\HH \uHH(\GG)\ar^-{\partial_{X,Z} \otimes Id}[r]
 & \uHH(U) \otimes_\HH \uHH(\GG)
}
$$
where $\epsilon$ is the symmetry isomorphism for the tensor product
of $\smod \HH$. The result then follows from the fact $\epsilon=-1$.\footnote{
Indeed it is well known that for any classes
 $x,y \in \HH^{n,m}(X) \times \HH^{s,t}(X)$, $xy=(-1)^{n+s}yx$.}

\noindent (2) For this point, we refer the reader to the proof of
 \cite[2.6.5]{Deg5}.
\end{proof}

\subsection{Coniveau spectral sequence}

\begin{num}
Recall from \cite[sec. 3.1.1]{Deg6} that
 a \emph{triangulated exact couple}
 in a triangulated category $\T$ is the data of bigraded
objects $D$ and $E$ of $\T$ and homogeneous morphisms 
between them
\begin{equation}
\label{hlg_exact_couple}
\xymatrix@=28pt{
D\ar^{(1,-1)}_\alpha[rr]
 &&  D\ar^{(0,0)}_/8pt/\beta[ld] \\
 & E\ar^{(-1,0)}_/-6pt/\gamma[lu] &
}
\end{equation}
with the bidegrees of each morphism indicated in the diagram and such that
 for any pair of integers $(p,q)$:
\begin{enumerate}
\item $D_{p,q+1}=D_{p,q}[-1]$,
\item the following sequence is a distinguished triangle:
$$
D_{p-1,q+1} \xrightarrow{\alpha_{p-1,q+1}} D_{p,q}
 \xrightarrow{\beta_{p,q}} E_{p,q}
 \xrightarrow{\gamma_{p,q}} D_{p-1,q}=D_{p-1,q+1}[1].
$$
\end{enumerate}
We can associate with that triangulated exact couple a
 \emph{differential} according to the formula:
$d=\beta \circ \gamma$.
 According to point (2) above, $d^2=0$ and we have defined a complex in $\T$:
$$
\hdots \rightarrow E_{p+1,q} \xrightarrow{d_{p,q}} E_{p,q} \rightarrow \hdots
$$
Recall also from \cite[def. 3.3]{Deg6} that, given a smooth scheme $X$,
 a \emph{flag} of $X$ is a decreasing sequence $(Z^p)_{p \in \ZZ}$ 
of closed subschemes of $X$ such that for all integer $p \in \ZZ$,
$Z^p$ is of codimension greater than $p$ in $X$ if $p \geq 0$
 and $Z^p=X$ otherwise.
We denote by $\drap X$ the set of flags of $X$,
 ordered by termwise inclusion.

Consider such a flag $(Z^p)_{p \in \ZZ}$. Put $U_p=X-Z^p$.
By hypothesis, we get an open immersion $j_{p}:U_{p-1} \rightarrow U_p$.
In the category of pointed sheaves of sets on $\sm$,
 we get an exact sequence:
$$
0 \rightarrow (U_{p-1})_+ \xrightarrow{j_{p}} (U_p)_+
 \rightarrow U_p/U_{p-1} \rightarrow 0
$$
because $j_{p}$ is a monomorphism (in the category of schemes).
Then, for any pair of integers $(p,q)$,
 we deduce from that exact sequence 
 a distinguished triangle in $\SH$:
\begin{equation} \label{eq:pre_coniv_ec}
\begin{split}
\sus (U_{p-1})_+[-p-q]
& \xrightarrow{j_{p*}} \sus (U_p)_+[-p-q]
 \xrightarrow{} \sus(U_p/U_{p-1})[-p-q] \\
& \xrightarrow{} \sus (U_{p-1})_+[-p-q+1],
\end{split}
\end{equation}
which in turn gives a triangulated exact couple according
 to the above definition.
\end{num}

\begin{num}
In the followings,
 we will not consider the preceding exact couple
 for only one flag. Rather, we remark that the triangle
 \eqref{eq:pre_coniv_ec} is natural with respect to the inclusion of flags:
 thus we really get a projective system of triangles
  and then a projective system of triangulated exact couples.
  
Recall that a pro-object of a category $\mathcal C$ 
is a (\emph{covariant}) functor $F$ from a left filtering 
category $\mathcal I$ to the category $\mathcal C$.
Usually, we will denote such a pro-object $F$ by the intuitive notation
$\pplim{i \in \mathcal I} F_i$
 and call it the \emph{formal projective limit} of
 the projective system $(F_i)_{i \in \mathcal I}$.

For any integer $p \in \ZZ$,
 we introduce the following pro-objects of $\SH$:
\begin{align}
F_p(X)&=\pplim{Z^* \in \drap X} \sus(X-Z^p)_+ \\
Gr_p(X)&=\pplim{Z^* \in \drap X} \sus(X-Z^p/X-Z^{p-1})
\end{align}
Taking the formal projective limit of the triangles
 \eqref{eq:pre_coniv_ec} where $Z^*$ runs over $\drap X$,
  we get a pro-distinguished triangle\footnote{
 \emph{i.e.} the formal projective limit of a projective system
 of distinguished triangles.}:
\begin{equation}  \label{eq:coniv_ec}
\begin{split}
F_{p-1}(X)[-p-q]
& \xrightarrow{\alpha_{p-1,q+1}} F_{p}(X)[-p-q]
 \xrightarrow{\beta_{p,q}} Gr_{p}(X)[-p-q] \\
& \xrightarrow{\gamma_{p,q}} F_{p-1}(X)[-p-q+1],
\end{split}
\end{equation}
\end{num}
\begin{df}
Considering the above notations, we define
 the \emph{homotopy coniveau exact couple} as data
 for any couple of integers $(p,q)$ of the pro-spectra:
$$
D_{p,q}=F_p(X)[-p-q], \qquad E_{p,q}=Gr_p(X)[-p-q]
$$
and that of the homogeneous morphisms of pro-objects $\alpha$, $\beta$,
 $\gamma$ appearing in the pro-distinguished triangle
 \eqref{eq:coniv_ec}.
\end{df}
For short, a projective system of triangulated exact couples
 will be called a \emph{pro-triangulated exact couple}.

\begin{ex}\label{eq:general_coniv_ssp}
Consider a spectrum $\E$.
We extend the functor represented by $\E$ to pro-spectrum as follows:
$$
\bar \varphi_\E:(\F_i)_{i \in \cI} \mapsto \ilim{i \in \cI^{op}} \Hom_{\SH}(\F_i,\E).
$$
Applying the funtor $\bar \varphi_\E$ to the homotopy coniveau exact couple
 gives an exact couple of abelian groups in the classical sense
 (with the conventions of \cite[th. 2.8]{McC}) whose associated spectral sequence
 is:
\begin{equation} \label{eq:coniv_ssp}
E_1^{p,q}(X,\E)=\ilim{Z^* \in \drap X} \E^{p+q}(X-Z^p/X-Z^{p-1})
 \Rightarrow \E^{p+q}(X).
\end{equation}
This is the usual coniveau spectral sequence associated with $\E$
 (see \cite{BO}, \cite{CTHK}).
 Note that it is concentrated in the band $0 \leq p \leq \dim(X)$;
 thus it is convergent.
\end{ex}

\begin{rem}
The canonical functor
$$
\SH \rightarrow \DM,
$$
once extended to pro-objects, sends the data defined above
 to the motivic coniveau exact couple considered in \cite[def. 3.5]{Deg6}.
\end{rem}

\section{The proof} \label{sec:proof}

\subsection{The (weak) $\HH$-module structure}

\begin{prop} \label{prop:orientable_hmod&wmod_HH}
Let $F_*$ be an orientable homotopy module. 

Then the spectrum $H(F_*)$ (see \eqref{eq:generators_hmod})
 admits a canonical structure of $\HH$-module in $\SH$.
\end{prop}
\begin{proof}
This follows from the exact sequence \eqref{eq:sec_Morel},
 example \ref{ex:unramified_Milnor&HH} and the fact that the
 tensor product on $\SH$ preserves positive objects
 for the homotopy t-structure.
\end{proof}

\begin{num}
In particular, the presheaf represented by 
 the (weak) $\HH$-module $H(F_*)$ precomposed with the functor
 \eqref{eq:strict_to_weak_mod} induces a canonical functor:
\begin{equation} \label{eq:hmod&smod_HH}
\varphi_F:(\smod \HH)^{op} \rightarrow \ab.
\end{equation}
According to the commutative diagram \eqref{eq:strict_to_weak_mod_ppty},
 we get a commutative diagram:
\begin{equation} \label{eq:hmod&smod_HH_ppty}
\begin{split}
\xymatrix@R=4pt@C=40pt{
\SH^{op}\ar^-{\varphi^0_F}[rd]\ar_{(L_\HH)^{op}}[dd] &  \\
& \ab \\
(\smod \HH)^{op}\ar_-{\varphi_F}[ru] &
}
\end{split}
\end{equation}
where $\varphi_F^0$ is the presheaf represented by 
 the spectrum $H(F_*)$.
 According to the isomorphism \eqref{eq:coh_hmod},
 this implies that for any smooth scheme $X$ and any integer 
 $n \in \ZZ$,
  $\varphi_F(\uHH(X)(n)[n])=F_{-n}(X)$.
\end{num}

\subsection{The associated cycle premodule}

\begin{num} \label{num:smooth_model}
Let $\cO$ be a formally smooth essentially of finite type $k$-algebra.
A \emph{smooth model} of $\cO$ will be an affine smooth scheme $X=\spec A$
 (of finite type) such that $A$ is a sub-$k$-algebra of $\cO$.
 Let $\model(\cO/k)$ be the set of such smooth models,
  ordered by the relation: $\spec B \leq \spec A$ if $A \subset B$.
 As $k$ is perfect, $\model(\cO/k)$ is a non empty left filtering ordered set.
 We will denote by $(\cO)$ the pro-scheme $(X)_{X \in \model(\cO/k)}$.
\end{num}
\begin{thm} \label{thm:generic_smod}
Consider the above notations and the category $\ecorps$
 introduced in \ref{num:axioms_pre_modl}.
There exists a canonical additive functor
$$
\uHH^{(0)}:\ecorps^{op} \rightarrow \pro{(\smod \HH)}
$$
defined on an object $(E,n)$ of $\ecorps$ by the formula:
$$
\uHH^{(0)}(E,n)=\pplim{X \in \model(E/k)} \uHH(X)(-n)[-n].
$$
\end{thm}
Note this theorem follows directly from our previous work
 on generic motives \cite{Deg5} when $k$
 admits resolution of singularities because the adjunction
 \eqref{eq:Hmod&DM} is then an isomorphism.
However, we give a proof
 (see paragraphs \ref{num:pf_thm:generic_smod_(D)}
  and \ref{num:relations}) 
 which avoids this assumption.
 It uses the same arguments than \cite{Deg5} after
 a generalisation of its geometric constructions 
 to the category $\smod \HH$ 
 (see section \ref{sec:recall_Gysin}).
 But let us first
  state the corollary which motivates the previous result:
\begin{cor} \label{cor:hmod&modl}
Let $F_*$ be an orientable homotopy module.

Then one associates with $F_*$ a canonical cycle premodule
$\hat F_*:\ecorps \rightarrow \ab$ defined on an object
$(E,n)$ of $\ecorps$ by the formula:
$$
\hat F_*(E,n)=\ilim{X \in \model(E/k)} F_{-n}(X).
$$
This defines a functor:
 $\tilde \rho':\hmor \rightarrow \pmodl$.
\end{cor}
\begin{proof}
In fact, the functor $\varphi_F$ associated with $F_*$
 in \eqref{eq:hmod&smod_HH} admits an obvious
 extension $\bar \varphi_F$ to pro-objects of $\smod \HH$.
 One simply puts: $\hat F_*=\bar \varphi_F \circ \uHH^{(0)}$.
 This is obviously functorial in $F_*$.
\end{proof}

\begin{num} \label{num:pf_thm:generic_smod_(D)}
\underline{Functorialities}: \\
\textbf{(D1)} is induced by the natural functoriality. \\
\textbf{(D2)}: A finite extension fields $\varphi:E \rightarrow L$
 induces a morphism of pro-schemes $(\varphi):(L) \rightarrow (E)$
 such that
$$
(\varphi)=\pplim{i \in I} (f_i:Y_i \rightarrow X_i)
$$
where the $f_i$ are finite surjective morphisms,
 whose associated generic residual extension is $L/E$,
 and the transition morphisms in the previous formal projective limit
 are made by transversal squares (see \cite[5.2]{Deg5} for details). 
 One defines the map $\varphi_*:\uHH(E) \rightarrow \uHH(L)$
 corresponding to (D2) as the formal projective limit:
$$
\pplim{i \in I} \big(f_i^*:\uHH(X_i) \rightarrow \uHH(Y_i)\big)
$$
using the Gysin morphism \eqref{eq:Gysin_morph} and property (G2c). \\
\textbf{(D3)}: According to example \ref{ex:unramified_Milnor&HH},
for any function field $E/k$ and any integer $n \geq 0$,
\begin{equation} \label{eq:cohm&Milnor}
\Bigg(\ilim{X \in \model(E/k)} \HH^n(X)\Bigg) \simeq K_n^M(E).
\end{equation}
Thus any symbol $\sigma \in K_n^M(E)$ corresponds
to a morphism of pro-objects
$$
\uHH^{(0)}(E) \rightarrow \uHH(n)[n]
$$
still denoted by $\sigma$.
For any smooth model $X$ of $E/k$,
 we let $\sigma_X:\uHH^{(0)}(X) \rightarrow \uHH(n)[n]$
 the component of $\sigma$ corresponding to $X$.
We define $\gamma_\sigma$ as the formal projective limit:
$$
\pplim{X \in \model(E/k)} \big(\sigma_X \ecuppx X 1_{X*}\big)
$$
with the definition given by formula \eqref{eq:product}. \\
\textbf{(D4)}: Let $(E,v)$ be a valued function field
 with ring of integers $\cO_v$ and residue field $\kappa_v$.
There exists a smooth model $X$ of $\cO_v$ and a point $x \in X$ 
 of codimension $1$ corresponding to the valuation $v$.
Let $Z$ be the reduced closure of $x$ in $X$.
Given an open neighborhood $U$ of $x$ in $X$ such that $Z \cap U$ is smooth,
we can write the corresponding Gysin triangle \eqref{eq:Gysin_tri_smod} as
follows:
\begin{equation} \label{eq:Gysin_particular}
\uHH(Z \cap U)(1)[1] \xrightarrow{\partial_{U,Z \cap U}} \uHH(U-Z \cap U)
 \xrightarrow{j_*} \uHH(U) \rightarrow \uHH(Z \cap U)(1)[2]
\end{equation}
According to property (G1a), the morphism $\partial_{U,Z \cap U}$
is functorial with
respect to the open subscheme $U$. Taking the formal projective limit
of this morphism with respect to the neighborhoods $U$ as above,
we obtain the desired map:
$$
\partial_v:\uHH(\kappa_v)(1)[1] \rightarrow \uHH(E).
$$
\end{num}

\begin{rem} \label{rem:pre_R3c}
Note for future references that the triangle \eqref{eq:Gysin_particular}
 being distinguished, we get with the above notations:
  $j_* \circ \partial_{U,Z \cap U}=0$.
\end{rem}

\begin{num}
Before proving the relations,
 we note that one can apply the preceding construction to obtain
 maps in motivic cohomology. More precisely, given any function fields $E/k$,
 we put (as in corollary \ref{cor:hmod&modl}):
$$
\hat \pi_0(\HH)_n(E):=\ilim{X \in \model(E/k)} \HH^n(X).
$$
As $\HH^n(X)=\Hom_{\smod \HH}(\HH(X)(-n)[-2n],\HH)$,
 we thus obtain that
  for any extension (resp. finite extension) of
   function fields $\varphi:K \rightarrow E$,
 the map (D1) (resp. (D2)) induces a canonical morphism:
\begin{align*}
&\varphi_*:\hat \pi_0(\HH)_n(K) \rightarrow \hat \pi_0(\HH)_n(E), \\
\text{resp. } & \varphi^*:\hat \pi_0(\HH)_n(E) \rightarrow \hat \pi_0(\HH)_n(K).
\end{align*}
We will need the following lemma concerning these maps:
\end{num}
\begin{lm} \label{lm:mot_coh_Milnor}
Consider for any function field $E/k$ the isomorphism
 \eqref{eq:cohm&Milnor} of graded abelian groups:
$$
K_*^M(E) \xrightarrow{\epsilon_E} \hat \pi_0(\HH)_*(E).
$$
Then, the following properties hold:
\begin{enumerate}
\item $\epsilon_E$ is an isomorphism of graded algebras.
\item For any morphism $\varphi:E \rightarrow L$ of function fields,
 the following square is commutative:
$$
\xymatrix@=24pt{
K_*^M(E)\ar^-{\epsilon_E}[r]\ar_{\varphi_*}[d]
 & \hat \pi_0(\HH)_*(E)\ar^{\varphi_*}[d] \\
K_*^M(L)\ar^-{\epsilon_L}[r] & \hat \pi_0(\HH)_*(L)
}
$$
where the left (resp. right) vertical map corresponds 
 to the standard functoriality of Milnor K-theory
 (resp. data (D1)).
\item For any finite morphism $\varphi:E \rightarrow L$ of function fields,
 the following square is commutative:
$$
\xymatrix@=24pt{
K_*^M(E)\ar^-{\epsilon_E}[r] & \hat \pi_0(\HH)_*(E) \\
K_*^M(L)\ar^-{\epsilon_L}[r]\ar^{\varphi^*}[u]
  & \hat \pi_0(\HH)_*(L)\ar_{\varphi^*}[u]
}
$$
where the left (resp. right) vertical map corresponds
 to the standard functoriality of Milnor K-theory
 (resp. data (D2)).
 \end{enumerate}
\end{lm}
\begin{proof}
Assertions (1) and (2) follow precisely from \cite[th. 3.4]{SV_BK}.
The verification of (3) is not easy due to the abstract nature
 of the Gysin morphism used to define (D2).
However, we can follow the argument of \cite{SV_BK}, lemmas 3.4.1
 and 3.4.4: this means we are reduced to prove the following formulas:
\begin{lm}
Let $\varphi:K \rightarrow E$
 and $\psi:K \rightarrow L$
 be finite extensions of fields.
 Let $[E:K]$ (resp. $[E:K]_i$) be the degree (resp. inseparable degree)
 of $E/K$.
 Then for any elements $x \in \hat \pi_0(\HH)_n(E)$ and
 $y \in \hat \pi_0(\HH)_n(K)$, the following formulas hold:
\begin{enumerate}
\item[(4)] $\varphi^*(x.\varphi_*(y))=\varphi^*(x).y$,
 $\varphi^*(\varphi_*(y).x)=y.\varphi^*(x)$.
\item[(5)] $\varphi^*\varphi_*(x)=[E:K].x$.
\item[(6)] Put $R=E \otimes_K L$. Then, 
 $\psi_*\varphi^*(x)=\sum_{z \in \spec R}
 \mathrm{lg}(R_z).\bar \varphi_{z}^*\bar \psi_{z*}(x)$,
 with the notations of (R1c).
\item[(7)] Assume $L/K$ is normal. Then,
 $\psi_*\varphi^*(x)=[E:K]_i.\sum_{j \in \Hom_K(E,L)} j_*(x)$.
\end{enumerate}
\end{lm}
In fact, point (7) implies point (3) for the graded part of degree $1$
 because according to point (2),
  $\psi_*:\pi_0(\HH)_1(K) \rightarrow \pi_0(\HH)_1(L)$
  is injective.
Then the proof of \cite[lem. 3.4.4]{SV_BK} allows to deduce
point (3): this proof indeed uses only the preceding fact together
with points (4) and (5).

Let us now prove the preceding lemma.
Points (4) (resp. (6)) follow from the definition of (D2)
 and property (G3b) (resp. (G2d)).
 Point (7) is then an easy consequence of (6) using elementary
 Galois theory. The difficult part is point (5).

According to point (4), we are reduced to prove that
 $\varphi^*(1)=[E:K]$.
Because of property (G2a) of the Gysin morphism,
 we reduce to the case where $E/K$ is generated by a single element.
 In other words, $E=K[t]/(P)$ where $P$ is a polynomial with coefficients in $t$.
 Thus we can find a smooth model $X$ (resp. $Y$) of $K/k$ (resp. $E/k$)
  and a finite surjective morphism $f:Y \rightarrow X$ 
  whose generic residual extension is $E/K$
 and which factors as:
$$
Y \xrightarrow i \PP^1_X \xrightarrow p X
$$
where $p$ is the canonical projection
 and $i$ is a closed immersion of codimension $1$.
 Thus point (5) now follows from the definition of (D2)
 and the trace formula \eqref{eq:trace_formula}.
\end{proof}

\begin{num} \label{num:relations}
\underline{Relations}: \\
The relations follow from the preparatory work done in section \ref{sec:recall_Gysin},
 according to the following table:
\begin{center}
\begin{tabular}{|r|l||r|l|}
\hline
 (R0) & \ref{lm:mot_coh_Milnor}(1)
& (R1a) & \text{Obvious.} \\
\hline
(R1b) & (G2a)
& (R1c) & (G2d) \\
\hline
(R2a) & (G3a)+lem. \ref{lm:mot_coh_Milnor}(2)
& (R2c) & (G3b)+lem. \ref{lm:mot_coh_Milnor}(2) \\
\hline
(R2c) & (G3b)+lem. \ref{lm:mot_coh_Milnor}(3)
& (R3a) & (G1b) \\
\hline
(R3b) & (G2b)
& (R3c) & rem. \ref{rem:pre_R3c} \\
\hline
(R3d) & prop. \ref{prop:residues_units}(2)
& (R3e) & prop. \ref{prop:residues_units}(1) \\
\hline
\end{tabular}
\end{center}
\end{num}
\noindent This concludes the proof of \ref{thm:generic_smod}.

\subsection{The associated cycle module}

\begin{num} \label{num:e_c_smodHH}
Using the obvious extension of the functor \eqref{eq:sh->Hmod} to pro-objects,
 the homotopy coniveau exact couple induces a pro-triangulated exact couple
 in $\pro{(\smod \HH)}$ whose graded terms are:
\begin{equation} \label{eq:e_c_smodHH}
\begin{split}
\uHH(F_p(X))[-p-q]&=\pplim{Z^* \in \drap X} \uHH(X-Z_p)[-p-q] \\
\uHH(Gr_p(X))[-p-q]&=\pplim{Z^* \in \drap X} \uHH(X-Z_p/X-Z_{p-1})[-p-q]
\end{split}
\end{equation}
Given a smooth closed subscheme $Z \subset X$
 of pure codimension $p$, if we apply \cite[prop. 4.3]{Deg8}
 to the closed pair $(X,Z)$ according to paragraph \ref{Deg8applies},
 we get a canonical isomorphism:
\begin{equation} \label{eq:purity}
\mathfrak p_{X,Z}:\uHH(X/X-Z) \rightarrow \uHH(Z)(p)[2p].
\end{equation}
Using this isomorphism, we can easily obtain
 the analog of \cite[3.11]{Deg6} in the setting of $\HH$-modules:
\end{num}
\begin{prop} \label{prop:comput_E_coniv}
Consider the above notations. Let $p \in \ZZ$ be an integer
 and denote by $X^{(p)}$ the set of
 codimension $p$ points of $X$.\footnote{
 Conventionally, it is empty if $p<0$ or $p>\dim(X)$.}
 Then there exists a canonical isomorphism:
$$
\uHH(Gr_p^M(X)) \xrightarrow{\quad \epsilon_p \quad }
 \pprod{x \in X^{(p)}} \uHH(\kappa(x))(p)[2p].
$$
\end{prop}
In particular, for any point $x \in X^{(p)}$ we get a canonical projection map:
\begin{equation} \label{eq:proj_e_c}
\pi_x:\uHH(Gr_p(X)) \rightarrow \uHH(\kappa(x))(p)[2p].
\end{equation}
For the proof, we refer the reader to the proof of \cite[3.11]{Deg6}
 --- the same proof works in our case
  if we use the purity isomorphism \eqref{eq:purity}.

\num \label{df:differentielles}
Let $X$ be a scheme essentially of finite type over $k$
 and consider a couple $(x,y) \in X^{(p)} \times X^{(p+1)}$.

Assume that $y$ is a specialization of $x$.
Let $Z$ be the reduced closure of $x$ in $X$
 and $\tilde Z \xrightarrow f Z$ be its normalization. 
Each point $t \in f^{-1}(y)$ corresponds to a
discrete valuation $v_t$ on $\kappa_x$ with residue 
field $\kappa_t$.
We denote by $\varphi_t:\kappa_y \rightarrow \kappa_t$ the 
morphism induced by $f$. Then, we define the following morphism 
of pro-$\HH$-modules:
\begin{equation} \label{eq:residues_pt}
d_y^x=
\sum_{t \in f^{-1}(y)}
 \partial_{v_t} \circ \varphi_{t*}:
\uHH(\kappa_y)(1)[1]
  \rightarrow \uHH(\kappa_x)
\end{equation}
using the notations of paragraph \ref{num:pf_thm:generic_smod_(D)}.
If $y$ is not a specialization of $x$,
  we put: $\partial_y^x=0$.
\begin{prop}
\label{compute_differentials}
Consider the above hypothesis and notations.
If $X$ is smooth then the following diagram is commutative:
$$
\xymatrix@C=45pt{
\uHH(Gr_{p+1}(X))\ar^{d_{p+1,-p-1}}[r]\ar_{\pi_y}[d]
 & \uHH(Gr_p^M(X))[1]\ar^{\pi_x}[d] \\
\uHH(\kappa_y)(p+1)[2p+2]
 \ar^-{d_y^x}[r]
  & \uHH(\kappa_x)(p)[2p+1]
}
$$
where the vertical maps are defined in \eqref{eq:proj_e_c}
 and the map $d_{p+1,-p-1}$ in the differential associated with
 the pro-triangulated exact couple \eqref{eq:e_c_smodHH}.
\end{prop}
\noindent The proof is the same as the one of \cite[3.13]{Deg6}
 once we use the following correspondence
 table for the results used in it:
\begin{center}
\begin{tabular}{|c|c|}
\hline
\cite[3.13]{Deg6} & \hspace{0.1cm} Proposition \ref{compute_differentials} \hspace{0.1cm} \\
\hline
Proposition 1.36 & Lemma \ref{lm:add_Gysin} \\
Theorem 1.34, relation (2) & (G2b) \\
Proposition 2.13 & Remark \ref{rem:G2b_plus} \\
Proposition 2.9 & (G2a) \\
\hline
\end{tabular}
\end{center}

\begin{num}
Consider an orientable homotopy module $F_*$,
 a smooth scheme $X$ and an integer $n \in \ZZ$.
 Put $\F=H(F_*)$ (using the functor of \eqref{eq:hmod_SH}). 

The coniveau spectral sequence
 \eqref{eq:coniv_ssp} associated with $\F(n)[n]$ has
 the following shape:
$$
E_1^{p,q}(X,\F(n)[n]) \Rightarrow H^{p+q}_\nis(X,F_n),
$$
where we have used the isomorphism \eqref{eq:coh_hmod} to identify
the abutment.

Consider also the commutative diagram \eqref{eq:hmod&smod_HH_ppty}
 and the obvious extension of $\bar \varphi_F^0$ (resp. $\bar \varphi_F$)
to pro-objects:
$$
\bar \varphi_F^0:\pro{\SH}^{op} \rightarrow \ab,
\text{ resp. } \bar \varphi_F:\pro{(\smod \HH)}^{op} \rightarrow \ab.
$$
Then the previous spectral sequence is defined by the exact couple
$$
\big(\bar \varphi_F^0(D_{p,q}),\bar \varphi_F^0(E_{p,q})\big)
=\big(\bar \varphi_F(L_\HH(D_{p,q})),\bar \varphi_F(L_\HH(E_{p,q}))\big).
$$
Thus, proposition \ref{prop:comput_E_coniv} gives a canonical
 isomorphism:
$$
E_1^{p,q}(X,\F(n)[n]) \xrightarrow{\epsilon_p^*} 
\begin{cases}
C^p(X,\hat F_*)_n & \text{if } q=0, \\
0 & \text{otherwise,}
\end{cases}
$$
where $\hat F_*$ is the cycle premodule associated with $F_*$.
Consider moreover, a couple $(x,y) \in X^{(p)} \times X^{(p+1)}$.
Comparing formula \eqref{eq:residues_pt_modl} with formula
 \eqref{eq:residues_pt}, proposition \ref{compute_differentials}
 gives the following commutative diagram:
$$
\xymatrix@C=40pt@R=20pt{
E_1^{p,0}(X,\F(n)[n])\ar^-*{d_1^{p,0}}[r] & E_1^{p+1,0}(X,\F(n)[n]) \\
\hat F_*(\kappa(x))\ar^{\partial_y^x}[r]\ar[u] & \hat F_*(\kappa(y))\ar[u]
}
$$
where the vertical maps are the canonical injections.
In particular:
\end{num}
\begin{cor} \label{cor:comput_coniv_ssp1}
Consider the previous notations.
\begin{enumerate}
\item The cycle premodule $\hat F_*$ satisfies properties
 \pFD X and \pC X.
\item There is a canonical isomorphism of complexes:
$$
E_1^{*,0}(X,\F(n)[n]) \simeq C^*(X;\hat F_*)_n.
$$
\item For any integer $p \in \ZZ$, there is a canonical isomorphism:
\begin{equation*} \label{eq:hmod&modl_coh&Chow}
H^p_\nis(X,F_n) \simeq A^p(X;\hat F_*)_n
\end{equation*}
where the right hand side is the $p$-th cohomology of the complex
 $C^*(X;\hat F_*)_n$ (notation of paragraph \ref{num:df_Chow_gp}).
\item The cycle premodule $\hat F_*$ is a cycle module.
\end{enumerate}
\end{cor}
The last property follows from lemma \ref{lm:reduce_modl} and point (1) above.
It implies in particular that the functor $\tilde \rho'$ of corollary \ref{cor:hmod&modl}
induces a canonical functor
$$
\tilde \rho:\hmor \rightarrow \modl
$$
establishing the first part of Theorem \ref{thm:conj_morel2}.

\subsection{The remaining isomorphism}

It remains to construct the natural isomorphism which appears in the statement
 of Theorem \ref{thm:conj_morel2}.
 Point (3) of the preceding corollary in the case $p=0$ gives
 an isomorphism of $\ZZ$-graded abelian groups:
$$
\epsilon_X:F_*(X) \rightarrow A^0(X;\hat F_*).
$$
According to the definition of the functor $A^0$
 of \eqref{eq:deglise},
 the right hand side is the sections over $X$ of the homotopy module
 $(\gamma_* \circ A^0 \circ \tilde \rho)(F_*)$.
We prove that $\epsilon_X$ is natural in $X$.
Explicitly, 
 for any morphism $f:Y \rightarrow X$ of smooth schemes,
 we have to prove the following diagram commutes:
\begin{equation} \label{eq:last_commutative_diag}
\begin{split}
\xymatrix@=20pt@R=14pt{
F_*(X)\ar^-{\epsilon_X}[r]\ar_{f^*}[d] & A^0(X;\hat F_*)\ar^{f^*}[d] \\
F_*(Y)\ar^-{\epsilon_Y}[r] & A^0(Y;\hat F_*)
}
\end{split}
\end{equation}
where the vertical map on the right hand side refers to the functoriality
 defined by Rost.

To prove this, we can assume by additivity that $X$ is connected,
 with function field $E$. According to our definition,
$$
\hat F_*(E)=\ilim{j_U:U \subset X} F_*(U)
$$
where the colimit runs over the non empty open subschemes of $X$.
In particular, the colimit of the morphism $j_U^*$
 induces a canonical map $\rho_X:F_*(X) \rightarrow \hat F_*(E)$.
 By definition of the coniveau spectral sequence,
  the isomorphism $\epsilon_X$ is induced by the exact sequence:
$$
0 \rightarrow F_*(X) \xrightarrow{\rho_X} \hat F_*(E)
 \xrightarrow{d^0_X} C^1(X;\hat F_*)_n
$$
where $d^0_X$ is the differential \eqref{eq:diff_modl}
associated with the cycle module $\hat F_*$.

\noindent (1) \textit{The case of a flat morphism}:
Consider a flat morphism $f:Y \rightarrow X$ of connected
 smooth schemes and let $\varphi:E \rightarrow L$ be the induced
 morphism on function fields.
 According to the definition of (D1) in paragraph \ref{num:pf_thm:generic_smod_(D)},
 the following square is commutative:
$$
\xymatrix@R=10pt@C=20pt{
F_*(X)\ar^-{\rho_X}[r]\ar_{f^*}[d] & \hat F_*(E)\ar^{\varphi_*}[d] \\
F_*(Y)\ar^-{\rho_Y}[r] & \hat F_*(L)
}
$$
Thus, the commutativity of \eqref{eq:last_commutative_diag}
 in this case follows from the definition of flat pullbacks on $A^0$
 (see \cite[12.2, 3.4]{Ros}).

\noindent (2) \textit{The general case}:
The morphism $f$ can be written as the composite
$$
Y \xrightarrow \gamma Y \times X \xrightarrow p X
$$
where $\gamma$ is the graph\footnote{As $X/k$ is separated by our general
 assumption, $\gamma$ is a closed immersion.} of $f$ and $p$ is the canonical projection.
To prove the commutativity of \eqref{eq:last_commutative_diag} in the case of $f$,
 we are reduced to the cases of $p$ and $\gamma$.
 The case of the smooth morphism $p$ follows from point (1).
Thus we are reduced to the case where $f=i:Z \rightarrow X$ is a closed
 immersion between smooth schemes.

Consider an open subscheme $U \subset X$ such that
 the induced open immersion $j_Z:Z \cap U \rightarrow Z$ is dense.
 Recall that the map $j_Z^*:A^0(Z;\hat F_*) \rightarrow A^0(Z \cap U;\hat F_*)$
  is injective\footnote{This follows
   for example from the localization long exact sequence
   for Chow groups with coefficients (\emph{cf.} \cite[3.10]{Ros}).}.
 Thus, according to case (1) when $f=j$ and $f=j_Z$,
  we are reduced to prove the commutativity of \eqref{eq:last_commutative_diag}
  when $f$ is the closed immersion $Z \cap U \rightarrow U$.

In particular, we can assume that $i$ is the composition of immersions
 of a smooth principal divisor. Then we are restricted to the case where
 $Z$ is a smooth principal divisor.

Because $Z$ is principal, it is parametrized by a regular function
 $\pi:X \rightarrow \AA^1$.
One can assume $Z$ is connected. Then, its generic point defines
 a codimension $1$ point of $X$ which corresponds to a discrete valuation $v$
 on the function field $E$ of $X$.
 We denote by $\kappa(v)$ the residue field of $v$.
 According to the computation of
 $i^*:A^0(X;\hat F_*) \rightarrow A^0(Z;\hat F_*)$ of \cite[(12.4)]{Ros},
 the commutativity of \eqref{eq:last_commutative_diag} in the case $f=i$
 is equivalent to the commutativity of the following diagram:
$$
\xymatrix@=20pt@R=14pt{
F_*(X)\ar^-{\rho_X}[r]\ar_{i^*}[d] & \hat F_*(E)\ar^{\partial_v \circ \gamma_{\pi}}[d] \\
F_*(Z)\ar^-{\rho_Z}[r] & \hat F_*(\kappa(v))
}
$$
But this follows from point (2) of Proposition \ref{prop:residues_units}
 and the definition of data (D3) and (D4) for the cycle module $\hat F_*$
 (paragraph \ref{num:pf_thm:generic_smod_(D)}).

According to the construction of the structural isomorphism
$$
A^0(X;\hat F_*)_n \rightarrow \big(A^0(.;\hat F_*)_{n+1}\big)_{-1}(X)
$$
given in \cite[2.8]{Deg9}, it is now clear that
 $\epsilon:F_* \rightarrow A^0(.\hat F_*)$ is a morphism of homotopy
modules. As it is an isomorphism by construction,
 this concludes the proof of the main theorem \ref{thm:conj_morel2}.
\section{Some further comments}

\subsection{Monoidal structure}

\begin{num}
Recall we have defined a $t$-structure on the category $\DM$ of motivic spectra
 (\emph{cf.} \cite[section 5.2]{Deg9})
 called the \emph{homotopy $t$-structure}
 whose heart is the category $\hmtr$ of definition \ref{df:hmtr}.
By the very construction, the right adjoint functor $\gamma_*$
of \eqref{eq:SH&DM} is $t$-exact with respect to the homotopy
$t$-structures on $\SH$ and $\DM$; thus it preserves the object of
the hearts. It also follows that its left adjoint
$\gamma^*$ preserves homologically positive objects
and \eqref{eq:SH&DM} induces an adjunction of abelian categories:
$$
\gamma^*_{\leq 0}:\hmod \leftrightarrows \hmtr:\gamma_*
$$
where $\gamma^*_{\leq 0}=t_- \gamma^*$.

According to this definition the functor $\gamma_*$ is equal to the composite map:
$$
\hmtr \xrightarrow{\gamma'_*} \hmor \rightarrow \hmod
$$
of the equivalence $\gamma'_*$ of Theorem \ref{thm:conj_morel1}
followed by the natural inclusion. Thus, it follows
from this later theorem that $\gamma_*$ is fully faithful.

Recall that the functor $\gamma^*$ is monoidal.
According to the definition of the tensor products on
$\hmod$ and $\hmtr$, it follows that the functor $\gamma^*_{\leq 0}$ is 
monoidal. Thus, its right adjoint is weakly monoidal:
 given any homotopy modules with transfers $F_*$ and $G_*$,
 we get a canonical comparison map of homotopy modules:
\begin{equation*} \label{eq:hmtr->hmod_monoidal}
\nu_{F,G}:\gamma_*(F_*) \otimes \gamma_*(G_*)
 \rightarrow \gamma_*(F_* \ohtr G_*)
\end{equation*}
where the tensor product $\otimes$ (resp. $\ohtr$) refers to the natural
 tensor product on $\hmod$ (resp. $\hmtr$)
 defined in \eqref{eq:hmod_tensor} (resp. \cite[1.17]{Deg9}).

The following result is a corollary of Theorem
 \ref{thm:conj_morel1}:
\end{num}
\begin{cor} \label{cor:gamma*&product}
For any homotopy modules with transfers $F_*$ and $G_*$,
 the morphism $\nu_{F,G}$ introduced above is an isomorphism.
\end{cor}
\begin{proof}
Because $\gamma_*$ is fully faithful, we can write $F_*=\gamma^*_{\leq 0}(F'_*)$
 (resp. $G_*=\gamma^*_{\leq 0}(G'_*)$) with $F'_*=\gamma_*(F_*)$
 (resp. $G'_*=\gamma_*(G_*)$).
Then we consider the following commutative diagram:
$$
\xymatrix@C=34pt{
\gamma_*\gamma^*_{\leq 0}(F'_*) \otimes \gamma_*\gamma^*_{\leq 0}(G'_*)
  \ar^-{\nu_{F,G}}[r]\ar_{a \otimes a'}[d]
 & \gamma_*(\gamma^*_{\leq 0}(F'_*) \ohtr \gamma^*_{\leq 0}(G'_*))\ar^-{\sim}[r] 
 & \gamma_*\gamma^*_{\leq 0}(F'_* \otimes G'_*)\ar^b[d] \\
F'_* \otimes G'_*\ar@{=}[rr] & & F'_* \otimes G'_*
}
$$
where $a$ (resp. $a'$, $b$) stands for the natural transformation
 $\alpha:1 \rightarrow \gamma_*\gamma^*_{\leq 0}$ evaluated at $F'_*$
 (resp. $G'_*$, $F'_* \otimes G'_*$).
Applying again the fact $\gamma_*$ is fully faithful, $\alpha_X$
is an isomorphism whenever $X$ belongs to the image of $\gamma_*$,
which is precisely $\hmor$. Thus $a$ and $a'$ are isomorphisms.
In order, to conclude it is sufficient to prove that $b$ is an isomorphism.
This amounts to show that the homotopy module $F_* \otimes G_*$ is orientable,
which is obvious from definition \ref{df:hmor}.
\end{proof}

\begin{num}
Recall that $\uK_*$, equipped with its canonical structure of a graded sheaf with transfers,
 is the unit object of the monoidal category $\hmtr$ (cf \cite[1.15, 3.8]{Deg9}).
As an object of $\hmod$, it corresponds to $\pi_0(\HH)_*$
 (example \ref{ex:unramified_Milnor&HH}).
Note that it is a commutative monoid in the monoidal category $\hmod$
 whose multiplication (resp. unit) map is equal to 
\begin{align*}
& \mu=\pi_0(\mu_\HH):\uK_* \otimes \uK_* \rightarrow \uK_* \\
\text{resp. } & u=\pi_0(u_\HH):\pi_0(S^0)_* \rightarrow \uK_*
\end{align*}
where $\mu_\HH$ (resp. $u_\HH$) is the multiplication (resp. unit)
 of the ringed spectrum $\HH$.
In particular, we can consider the category $\smod{\uK_*}$ of modules over $\uK_*$ 
 in the monoidal category $\hmod$: it is naturally a Grothendieck abelian monoidal
 category.
As a corollary of the previous result we get the following nice characterization of
orientability for homotopy modules:
\end{num}
\begin{cor} \label{cor:carac_mhtpo}
Consider the above notations. 
\begin{enumerate}
\item The multiplication map $\mu$ is an isomorphism.
\item Given a homotopy module $F_*$, the following conditions are equivalent:
\begin{enumerate}
\item[(i)] $F_*$ is an orientable homotopy module (definition \ref{df:hmor}).
\item[(ii)] $F_*$ as a Nisnevich sheaf admits a canonical structure of
 a sheaf with transfers.
\item[(iii)] The map $u \otimes 1_{F_*}:F_* \rightarrow \uK_* \otimes F_*$
 is an isomorphism.
\item[(iv)] $F_*$ admits a structure of a $\uK_*$-module in $\hmod$.
\item[(v)] The spectrum $H(F_*)$ has a structure of a (strict) $\HH$-module.
\item[(vi)] The spectrum $H(F_*)$ is orientable.
\end{enumerate}
\item The functor $\gamma^*_{\leq 0}:\hmod \rightarrow \hmtr$ can be factorized
as:
$$
\hmod \xrightarrow{\uK_* \otimes ?} \smod{\uK_*}
 \xrightarrow{\ \epsilon \ } \hmtr
$$
and $\epsilon$ is an equivalence of abelian monoidal categories.
\end{enumerate}
\end{cor}
\begin{proof}
Point (1) follows from Corollary \ref{cor:gamma*&product}
 applied in the case $F=G=\uK_*$. The assertion of point (2) follows
 from our previous results:
 (i) $\Rightarrow$ (ii): \ref{thm:conj_morel1},
 (ii) $\Rightarrow$ (iii): Point (1) follows from Corollary,
 (iii) $\Rightarrow$ (iv): obvious,
 (iv) $\Rightarrow$ (v): because $\uK_*=\pi_0(\HH)_*$ (see also
 the proof of \ref{prop:orientable_hmod&wmod_HH}),
 (v) $\Rightarrow$ (vi): using the morphism of ringed spectra
 $\MGL \rightarrow \HH$ corresponding to the orientarion of $\HH$,
 (vi) $\Rightarrow$ (i): see remark \ref{rem:mhor&P^1}.
 According to point (1), $\smod{\uK_*}$ is a full subcategory of $\hmod$ ;
 thus point (2) implies point (3).
\end{proof}
\begin{rem}
It is not true in general that the multiplication map
 $\mu:\HH \wedge \HH \rightarrow \HH$
 is an isomorphism according to the non triviality of the Steenrod operations.\footnote{
 However, we prove in \cite[13.1.6]{CD3} that $\mu \otimes_\ZZ \QQ$ is an isomorphism 
 (in $\SH_\QQ$).}
\end{rem}

\subsection{Weakly orientable spectra}

\begin{num}
Consider a spectrum $\E$.
The Hopf map $\eta$ obviously acts on the $\E$-cohomology of a smooth scheme
as:
$$
\gamma_\eta:\E^{n,p}(X) \rightarrow \E^{n-1,p-1}(X), \rho \mapsto \eta \wedge \rho.
$$
The following lemma is obvious:
\end{num}
\begin{lm}
Given a spectrum $\E$, the following conditions are equivalent:
\begin{enumerate}
\item[(i)] For any integer $n \in \ZZ$, $\pi_n(\E)_*$ is an orientable.
\item[(ii)] The map $\eta \wedge 1_\E$ is zero.
\item[(ii')] The map $\uHom(\eta,\E)$ is zero.
\item[(iii)] The operation $\gamma_\eta$ on $\E$-cohomology defined above is zero.
\end{enumerate}
\end{lm}

\begin{df}
A spectrum $\E$ satisfying the equivalent conditions of
 the preceding lemma will be said to be \emph{weakly orientable}.
\end{df}

\begin{rem}
From the second point of remark \ref{rem:mhor&P^1},
 if $2$ is invertible in $E^{**}$ and $(-1)$ is a sum of squares in $k$,
 then $\E$ is weakly orientable.
\end{rem}

As a corollary of our main theorem \ref{thm:conj_morel1},
 we get the following way to construct cycle modules:
\begin{prop}
Let $\E$ be a weakly orientable spectrum.

Then for any integer $i \in \ZZ$,
 there exists a canonical cycle modules whose values on a couple $(L,n)$
 of $\ecorps$ is the abelian group:
$$
\E^{i+n,n}(L)=\ilim{X \in \model(E/k)^{op}}\E^{i+n,n}(X),
$$
with the notations of paragraph \ref{num:smooth_model}.
\end{prop}
The new example here is given by the case where
 $\E=\MGL$ is Voevodsky's algebraic
 cobordism spectrum.

\subsection{Coniveau spectral sequence}

As a corollary of the detailed analysis of Proposition \ref{compute_differentials},
 we get the following statement:
\begin{prop}
Let $\E$ be a spectrum and $q \in \ZZ$ be an integer
 such that the homotopy module $\pi_{-q}(\E)$ is orientable.

Then for any smooth scheme $X$,
 there exists a canonical isomorphism of complexes
 of abelian groups:
$$
E_1^{*,q}(X,\E) \simeq C^*(X,\hat \pi_{-q}(\E)_*)_0
$$
where the left hand sides refers to the $q$-th line of the first page of
 the coniveau spectral sequence associated with $\E$
  --- \emph{cf.} example \ref{eq:general_coniv_ssp}.
\end{prop}
\begin{proof}
We can obviously assume $q=0$.
The complex $E_1^{*,0}(X,\E)$ is natural in $\E$. Moreover,
 for any spectrum $\F$, one checks using the purity isomorphism
 in $\SH$ (\cite[\textsection 4, 2.23]{MV})
 and the same argument as in the proof of
 Proposition \ref{prop:comput_E_coniv} that the induced maps:
$$
E_1^{*,0}(X,\F_{\geq 0}) \rightarrow E_1^{*,0}(X,\F), \quad
E_1^{*,0}(X,\F) \rightarrow E_1^{*,0}(X,\F_{\leq 0})
$$
are isomorphisms. This yields an isomorphism of complexes:
$$
E_1^{*,0}(X,\E) \simeq E_1^{*,0}(X,\pi_0(\E)).
$$
so that we are reduced to the case where $E$ is an orientable homotopy module.
This case is precisely Corollary \ref{cor:comput_coniv_ssp1}.
\end{proof}

\begin{rem} \label{num:coniv&hyper_coh}
The previous corollary can be applied to any weakly orientable spectrum $\E$:
 we get for any smooth scheme $X$ and any integer $n \in \ZZ$ a convergent spectral sequence
 of the form:
\begin{equation} \label{eq:ssp_coniv}
E_2^{p,q}(X,\E(n))=A^p(X,\hat \pi_{q-n}(\E)_*)_n \Rightarrow \E^{p+q,n}(X).
\end{equation}
\end{rem}

\begin{num} \label{num:MGH_et_H}
Let $\MGL$ be Voevodsky's algebraic cobordism spectrum.
According to a theorem of Morel, there exists a canonical isomorphism of homotopy
 modules: $\pi_0(\MGL)_* \simeq K_*^M$.\footnote{Indeed,
 this follows from the exact sequence
 \eqref{eq:sec_Morel}.}
One deduces from this result that the ring morphism $\varphi:\MGL \rightarrow \HH$
 corresponding to the orientation of $\HH$ given by \eqref{eq:Pic&motivic_coh}
 induces an isomorphism:
$$
\pi_0(\MGL)_* \rightarrow \pi_0(\HH)_*.
$$
Moreover, given a smooth connected scheme $X$ and an integer $d \in \ZZ$,
 the morphism $\varphi$ induces
 a morphism of the spectral sequences of type \eqref{eq:ssp_coniv} which, 
 on the $E_2$-term is equal to:
$$
A^p(X,\hat \pi_{q-d}(\MGL)_*)_d \rightarrow A^p(X,\hat \pi_{q-d}(\HH)_*)_d.
$$
In case $d$ is the dimension of $X$, the terms $E_2^{p,q}$ of
 the respective spectral sequences 
 for $\MGL$ and $\HH$ are concentrated in the region
 $0 \leq p \leq d$, $g \leq d$.
Thus we get the following:
\end{num}
\begin{cor} \label{cor:MGL_CH}
Let $X$ be a smooth scheme of pure dimension $d$.
Then the morphism $\varphi$ induces an isomorphism:
$$
\MGL^{2d,d}(X) \xrightarrow{\varphi_*} \HH^{2d,d}(X)=A^d(X,K_*^M)_d=CH^d(X).
$$
\end{cor}
Obviously, this isomorphism is natural with respect to pullbacks.
It is also compatible with pushouts induced by the Gysin morphism
 of a projective morphism $f:Y \rightarrow X$ of smooth connected schemes.

\begin{rem}
If $X$ is projective smooth, the corollary can be reformulated using duality
 (see \cite[th. 5.23]{Deg8}) 
 by saying that $\varphi$ induces an isomorphism in homology:
$$
\MGL_0(X) \xrightarrow{\sim} CH_0(X).
$$
If $X$ is only smooth,
 under the assumption of resolution of singularities,
 one can replace the left hand side in the previous isomorphism by
 the Borel Moore algebraic cobordism:
$$
\MGL^c_0(X) \xrightarrow{\sim} CH_0(X).
$$
\end{rem}

\subsection{Cohomology spectral sequence and cycle classes}

\begin{num}
Consider a spectrum $\E$.
 The truncation functor for the homotopy $t$-structure
 gives a canonical (functorial) tower in $\SH(k)$:
$$
\hdots \rightarrow \E_{\geq p} \rightarrow \E_{\geq p+1} \rightarrow \hdots
$$
called the \emph{Postnikov tower} of $\E$.

Then for any smooth scheme $X$ and any integer $n \in \ZZ$,
 the Postnikov tower of $\E(n)$ gives the following spectral sequence:
\begin{equation} \label{eq:ssp_t}
E_{2,t}^{p,q}(X,\E)=H^q_{\nis}(X,\pi_{n-p}(\E)_n) \Rightarrow \E^{p+q,n}(X)
\end{equation}
which we simply call the \emph{cohomology spectral sequence}.

Note that, when $\E$ is weakly orientable, Corollary
 \ref{cor:comput_coniv_ssp1}
 gives a canonical isomorphism between the $E_2$-term
  of the coniveau spectral sequence \eqref{eq:ssp_coniv}
 and the $E_2$-term of the cohomology spectral sequence. 
\end{num}

\begin{rem}
Using the same argument as in the proof of \cite[th. 6.4]{Deg9}
 and a construction of Gillet and Soul\'e\footnote{More precisely,
 ore replaces the use of the shifted filtration of Deligne on complexes  
 by its generalization for spectra given in \cite{GS}.},
 one can show that this isomorphism is compatible with the differential 
 on each $E_2$-term and that they induces an isomorphism of spectral sequences.
 However, in the remainings,
  we will only need to be able to compute the $E_2$-term using
  the isomorphism of Corollary \ref{cor:comput_coniv_ssp1}.
\end{rem}

\begin{num} \label{num:comp_ssp_t}
The advantage of the spectral sequence \eqref{eq:ssp_t}
 compared to the coniveau spectral sequence
 is that it is obviously functorial in $X$ (contravariantly).

Assume that $\E$ has the structure of a weak $\MGL$-module.
Consider smooth schemes $X$, $Y$ and a projective morphism $f:Y \rightarrow X$
 equidimensional of dimension $n$.
 Then according to the construction of  \cite[def. 5.12]{Deg8},
 we get a canonical morphism of $\MGL$-modules:
$$
f^*:\uMGL(X)(n)[2n] \rightarrow \uMGL(Y).
$$
Using diagram \eqref{eq:strict_to_weak_mod_ppty} in the case $R=\MGL$, 
 this map induces after applying the functor $\E^{i,j}$ a canonical pushout:
\begin{equation} \label{eq:Gysin_coh}
\E^{i,j}(Y) \xrightarrow{f_*} \E^{i-2n,j-n}(X).
\end{equation}
\end{num}
\begin{lm} \label{prop:cov_ssp_t}
Consider the notations above.
Then the Gysin map $f^*$ induces a morphism of spectral sequences:
$$
E_{2,t}^{p,q}(\uMGL(Y),\E)=H^q_{\nis}(Y,\pi_{-p}(\E)_0)
 \rightarrow H^{q-n}_{\nis}(Y,\pi_{-p}(\E)_{-n})=E_{2,t}^{p,q}(\uMGL(X)(n)[2n],\E)
$$
which converges to the morphism \eqref{eq:Gysin_coh}.
\end{lm}

Using the cohomology spectral sequence,
 we get the following proposition which gives mild conditions
 on a spectrum $\E$ for the existence of cycle classes in
 $\E$-cohomology satisfying the usual properties:
\begin{prop} \label{prop:eta_cycle_class}
Let $\E$ be a (weak) ring spectrum such that:
\begin{enumerate}
\item[(a)] The homotopy module $\pi_0(\E)_*$ is orientable.
\item[(b)] For any function field $K/k$, $\E^{n,m}(K)=0$ 
 if $n<0$ and $m<0$.
\end{enumerate}
Then the following conditions hold:
\begin{enumerate}
\item The spectrum $\E$ admits an orientation whose associated
 formal group law is additive.
\item For any smooth scheme $X$ and any integer $n \geq 0$,
 there exists a canonical morphism of abelian groups:
$$
\sigma_X:CH^n(X) \rightarrow \E^{2n,n}(X)
$$
which is natural with respect to pullbacks, projective pushforwards
 and compatible with pro\-ducts. 
\end{enumerate}
\end{prop}
\begin{proof}
According to assumption (a) and Corollary \ref{cor:carac_mhtpo},
 the unit map $S^0 \rightarrow \E$ induces a canonical map
$$
\uK_*=\uK_* \otimes \pi_0(S^0)_*
 \rightarrow \uK_* \otimes \pi_0(\E)_*=\pi_0(\E)_*
$$
which is a morphism of monoids in $\hmor$. In particular,
 according to Corollary \ref{cor:comput_coniv_ssp1},
 we get for any smooth scheme $X$
  and any integer $n \geq 0$ a canonical morphism:
\begin{equation} \label{eq:cycles1}
CH^n(X)=A^n(X,K_*)_n
 \rightarrow A^n(X,\hat \pi_0(\E)_*)_n=H^n_{\nis}(X,\pi_0(\E)_n)
\end{equation}
compatible with pullbacks, pushouts and products.

Applying again Corollary \ref{cor:comput_coniv_ssp1},
 assertion (b) implies that
 the term $E_{2,t}^{p,q}(X,\E(n))$ of the cohomology spectral sequence  
 is zero if $p>min(q,n)$.
Thus, we get a canonical composite map:
\begin{equation} \label{eq:cycles2}
E_{2,t}^{n,n}(X,\E(n))
 \xrightarrow{(1)} E_{\infty,t}^{n,n}(X,\E(n))
 \xrightarrow{(2)} \E^{2n,n}(X)
\end{equation}
where the map (1) is obtain as the sequence of epimorphism deduced
from the spectral sequence \eqref{eq:ssp_t} and the map (2)
is the edge morphism (which is a monomorphism).
This composite map is compatible with products and pushouts
 according to paragraph \ref{num:comp_ssp_t}.
The fact it is compatible with products follows from the
 construction of products on the spectral sequence of
 the type \eqref{eq:ssp_t} (see \cite{McC}).

The composition of \eqref{eq:cycles1} and \eqref{eq:cycles2}
 gives the map of the point (2). In the case $n=1$,
 we get an orientation of $\E$ whose associated formal group
 law is additive because $\sigma_X$ is a morphism of groups.
\end{proof}

\begin{rem}
According to a result announced by Morel and Voevodsky,
 at least over a field of characteristic $0$,
 the map induced by the morphism $\varphi$ of paragraph \ref{num:MGH_et_H}
 induces an isomorphism of ring spectra
$$
\MGL/\{a_{ij}, (i,j)\in \NN^2\} \rightarrow \HH
$$
where 
 $a_{ij}:S^{2(i+j),i+j} \rightarrow \MGL$
 denotes the coefficients of the formal group law of
 $\MGL$ equipped with its obvious orientation.

The orientation of the first point of the previous proposition
 corresponds to a morphism of ring spectra
$$
\psi:\MGL \rightarrow \E
$$
such that $\psi \circ a_{ij}$ is zero.
In particular, $\psi$ induces a canonical morphism of ring spectra:
$$
\sigma:\HH \rightarrow \E.
$$
which gives back the cycle class of point (2).
\end{rem}

\bibliographystyle{amsalpha}
\bibliography{mhtpo}

\end{document}